\documentclass{amsart}
\usepackage[latin1]{inputenc}
\usepackage{fancyhdr}
\usepackage{indentfirst}
\usepackage[dvips]{graphicx}
\usepackage{graphics}
\usepackage{newlfont}
\usepackage{amssymb}
\usepackage{amsmath}
\usepackage{amscd}
\usepackage{latexsym}
\usepackage{amsthm}
\usepackage{psfrag}
\usepackage[usenames]{color}
\usepackage{textcomp}
\usepackage[all]{xy}
\xymatrixrowsep={2pt}
\xymatrixcolsep={2pt}

\newcommand{\R}{\mathbb R}

\newcommand{\eps}{\varepsilon}

\newcommand{\m}{\preceq}
\newcommand{\Imm}{\mathrm{Im}}

\newenvironment{definition}[1][Definition]{\begin{trivlist}
\item[\hskip \labelsep {\bfseries #1}]}{\end{trivlist}}
\newenvironment{remark}[1][Remark]{\begin{trivlist}
\item[\hskip \labelsep {\bfseries #1}]}{\end{trivlist}}
\newenvironment{proposition}[1][Proposition]{\begin{trivlist}
\item[\hskip \labelsep {\bfseries #1}]}{\end{trivlist}}
\newenvironment{lemma}[1][Lemma]{\begin{trivlist}
\item[\hskip \labelsep {\bfseries #1}]}{\end{trivlist}}
\newenvironment{theorem}[1][Theorem]{\begin{trivlist}
\item[\hskip \labelsep {\bfseries #1}]}{\end{trivlist}}

\begin{document}

\title[Estimating Multidimensional Persistent Homology]{Estimating Multidimensional Persistent Homology through a Finite Sampling}

\author{Niccol\`o Cavazza}
\address{
niccolo.cavazza@gmail.com}

\author{Massimo Ferri}
\address{
Dipartimento di Matematica\\
Universit\`a di Bologna\\
Italia\\
massimo.ferri@unibo.it}

\author{Claudia Landi}
\address{Dipartimento di Scienze e Metodi dell'Ingegneria\\
Universit\`a di Modena e Reggio Emilia\\
Italia\\
clandi@unimore.it}

\begin{abstract}
An exact computation of the persistent Betti numbers of a submanifold $X$ of a Euclidean space
is possible only in a theoretical setting. In practical situations, only a finite
sample of $X$ is available. We show that, under suitable density conditions,
it is possible to estimate the multidimensional persistent Betti numbers of $X$ from the ones of
a union of balls centered on the sample points; this even yields the exact
value in restricted areas of the domain.

Using these inequalities we improve a previous lower bound for the  natural pseudodistance to assess
dissimilarity between the shapes of two objects from a sampling of them.

Similar inequalities are proved for the multidimensional persistent Betti numbers of the ball union and the one
of a combinatorial description of it.

\keywords{Persistent Betti Numbers, Ball covering, Voronoi diagram, Blind strip}

\end{abstract}

\maketitle

\section{Introduction}\label{Intro}

Persistent Topology is an innovative way of matching topology and geometry, and
it proves to be an effective mathematical tool in Pattern Recognition~\cite{Gh08,Ca09}, particularly in
Shape Comparison~\cite{BiCeFrGiLa}. This new research
area is experiencing a period of intense theoretical progress, particularly in the
form of the multidimensional {\em persistent Betti numbers} (PBNs; also called {\em rank invariant} in~\cite{CaZo09}).
In order to express its full potential for applications,
it has to interface with the typical environment of Computer Science: It must be possible to deal
with a finite sampling of the object of interest, and with combinatorial representations of it.

A predecessor of the PBNs, the {\em size function} (i.e. PBNs at degree zero)
already enjoys such a connection, in that it is possible to estimate it from a finite, sufficiently dense
sampling~\cite{Fr92}, and it is possible to simplify the computation by processing a related graph~\cite{dA00}. Moreover,
strict inequalities hold only in ``blind strips'', i.e. in the $\omega$-neighborhood of the discontinuity lines,
where $\omega$ is the modulus of continuity of the filtering (also called {\it measuring}) function. Out of the
blind strips, the values of the size function of the original object, of a ball covering of it,
and of the related graph coincide. As a consequence it is possible to estimate
dissimilarity of the shape of two objects by using size functions of two
samplings.

The present paper extends this result to the PBNs of any degree, by using an article by
P. Niyogi, S. Smale, and S. Weinberger~\cite{NiSmWe08}, for a ball covering of the object $X$ --- a submanifold of
a Euclidean space --- with balls centered at dense enough points of $X$ or near $X$: Theorems \ref{wrap} and \ref{wrap2}.
A combinatorial representation of $X$, with the corresponding inequalities (Theorem \ref{t5}) is
based on a construction by H. Edelsbrunner~\cite{Ed95}. In Section \ref{comparison} we use Theorem \ref{wrap} to get a lower bound
for the natural pseudodistance between two objects of which only finite samplings are given.

All results are provided for multidimensional filtering functions, although most current applications
just use monodimensional ones.

It should be noted that the same kind of problem has been addressed in~\cite{ChaLieu05} by using the notion of Weak Feature Size.
An inequality rather similar to ours of Theorem \ref{wrap}, also extending the main result of~\cite{Fr92}, is exposed in the
proof of the Homology Inference Theorem in Section 4 of~\cite{CoEdHa07}. The progress represented by the present paper
consists in the stress on inequalities (and not just on the consequent equalities), in the use of multidimensional
filtering functions, in the fact that any continuous function --- not just distance --- is considered, and
in the estimate of the natural pseudodistance.

Some simple examples illustrate the results.

\section{Preliminary results}\label{BasicResult}

In this section, we report some results on the stability of the
PBNs and some topological properties of compact
Riemannian submanifolds of $\mathbb{R}^m$. For basic notions on homology
and persistent homology we refer the reader to~\cite{EiSt} and~\cite{EdHa08},
for classical properties of submanifolds to~\cite{Hi76}.

\subsection{Multidimensional persistent Betti numbers}\label{stab}

In this paper we will always work with coefficients in a field
$\mathbb{K}$, so that all homology modules are vector spaces.
First we define the following relation $\prec$ (resp.
$\preceq$) in $\mathbb{R}^n$: if $u=(u_1,\ldots,u_n)$ and $v=(v_1,\ldots,v_n)$, we write $ u \prec  v$ (resp. $ u
\preceq  v$) if and only if $u_j<v_j$ (resp. $u_j \leq
v_j$) for $j=1,\ldots,n$. We also define $\Delta^+$ as the open set
$\{( u, v)\in\mathbb{R}^n\times\mathbb{R}^n \,|\, u\prec v\}$.\\
As usual, a topological space $X$ is $triangulable$ if there is a
finite simplicial complex whose underlying space is homeomorphic
to $X$; for a submanifold of a Euclidean space, it will mean that
its triangulation can be extended to a domain containing it. We
use \v Cech homology because it guarantees some useful continuity
properties~\cite{CeDiFeFrLa13}. For terms and concepts concerning this homology,
we refer to~\cite{EiSt}.

\begin{definition}[Persistent Betti numbers]\label{rango}
Let $X$ be a triangulable space and $ f=(f_1, \ldots, f_n):X\rightarrow
\mathbb{R}^{n}$ be a continuous function. $ f$ is called a
{\em filtering function}. We denote by $X\langle
 f \m  u \rangle$ the lower level subset $\{p\in X \,|\, f_j(p)\leq
u_j, j=1,\ldots, n\}$. Then, for each $i\in \mathbb{Z}$, the {\em $i$-th
multidimensional persistent Betti numbers (briefly PBNs) function of $(X, f)$} is $\beta_{(X,
f,i)}:\Delta^{+}\rightarrow \mathbb{N}$ defined as
$\beta_{(X,f,i)}( u, v)=\dim(\Imm \iota)$, with

\[\iota:\check{H}_{i}(X\langle  f \m
 u \rangle)\to \check{H}_{i}(X\langle  f \m
 v \rangle),\]

\noindent the homomorphism induced by the inclusion map of the sublevel set
\hbox{$X\langle  f \m  u \rangle \subseteq  X\langle  f \m  v \rangle$}.
Here $\check{H}_i$ denotes the $i$-th \v Cech homology module.
\end{definition}

\begin{remark}
When PBNs were considered in~\cite{CoEdHa07} the filtering functions
were taken as to be tame, but in our case, since we need continuous
maps, we refer to Section 2.2 of~\cite{CeDiFeFrLa13} for the relevant extension.
\end{remark}

PBNs give us a way to analyze triangulable spaces
through their homological properties. Then it is natural to
introduce a distance for comparing them. This has been done, and
through this distance it has been possible to prove stability of the PBNs
under variations of the filtering function in the one-dimensional~\cite{CoEdHa07}
and multidimensional case~\cite{CeDiFeFrLa13}.

\subsection{Topological properties of compact Riemannian submanifolds of $\mathbb R^m$}

As hinted in Section \ref{Intro}, we are interested in getting information
on a submanifold of $\mathbb R^m$ via a finite sampling of it and a related
ball covering. To this goal, we now state some properties of compact Riemannian submanifolds of
$\mathbb R^m$, especially referred to such an approximating covering. Definition
\ref{condition} and Proposition \ref{smale} are due to
P. Niyogi, S. Smale, and S. Weinberger~\cite{NiSmWe08}. The main
idea is that, under suitable hypotheses, it is possible to
get, from a sampling of a submanifold, a ball covering whose union retracts on it.\\
First, notice that the spaces $\mathbb{R}^m$ and $\mathbb{R}^n$ play two different r\^oles in our arguments:
the ambient space of our submanifolds (which will always be $\mathbb{R}^m$) is endowed with the
classical Euclidean norm and has no partial order relation on it.
On the other hand, the codomain of the filtering functions ($\mathbb{R}^n$ throughout)
is endowed with the max norm and with the partial order relation $\preceq$, as
defined at the beginning of Section \ref{stab}.

\begin{definition}[Modulus of continuity]\label{modulo}
Let $ f:\mathbb R^m\to \mathbb R^n$ be a
continuous function. Then, for $\eps\in \mathbb R^+$, the
{\em modulus of continuity} $\Omega(\eps)$ of $ f$ is:

\[\Omega(\eps)=\max_{j=1,\ldots, n}\sup
\Big\{\mathrm{abs}(f_j( p)-f_j( p'))\ |\   p,
p'\in \mathbb R^m,\  \| p- p'\|\leq
\eps \Big\}.\]

In other words $\Omega(\eps)$ is the maximum over all moduli of
continuity of the single components of $ f$.
\end{definition}

For a given compact Riemannian submanifold $X$ of $\mathbb R^m$, the normal space of $X$ at
a point $p \in X$ is the vector subspace $N_pX$ of the tangent space of $\mathbb R^m$ at $p$, $T_p\mathbb R^m$, formed by the vectors
orthogonal to $T_pX$, the tangent space of $X$ at $p$. The {\em open normal bundle to $X$ of radius $s$} is defined as the subset
of $T\mathbb R^m$
\[ \{(p, v) \in T\mathbb R^m \,\vert\, p\in X, \ v\in N_pX, \ ||v|| < s\} \]
By Theorem 10.19 of~\cite{Lee} and by compactness, there exists an embedding of the open normal bundle to $X$ of radius
$s$ into $\mathbb R^m$ for some $s$. Its image $Tub_sX$ is called a {\em tubular neighborhood} of $X$.


A {\it condition number} $\frac{1}{\tau}$ is associated with any
compact Riemannian submanifold $X$ of $\mathbb R^m$.

\begin{definition}\label{condition}
$\tau$ is the largest number such that every open normal bundle
$B$ about $X$ of radius $s$ is embedded in $\mathbb{R}^m$ for
$s<\tau$.

\end{definition}

\begin{proposition} (Prop. 3.1 of~\cite{NiSmWe08})\label{smale}
Let $X$ be a compact Riemannian submanifold of $\mathbb R^m$. Let
$L=\{l_1,\ldots,l_k\}$ be a collection of points of $X$, and let
$U=\bigcup_{j=1,\ldots, k}B(l_j,\delta)$ be the union
of balls of $\mathbb R^m$ with center at the points of $L$ and
radius $\delta$. Now, if $L$ is such that for every point $p\in X$
there exists an $l_j\in L$ such that $\|p-l_j\| <
\dfrac{\delta}{2}$, then, for every $\delta <
\sqrt{\dfrac{3}{5}}\tau$, $X$ is a deformation retract of
$U$. So they have the same homology.
\end{proposition}

\begin{remark}\label{rem}
The proof of Proposition \ref{smale} gives us a way to construct a retraction
$\pi:U\to X$ and a homotopy $F:U\times I \to U$
such that $F(q,0)=q$ and $F(q,1)=\pi(q)$.\\
Let $\pi_0:Tub_{\tau}\to X$ be the canonical
projection from the tubular neighborhood of radius $\tau$ of $X$
onto $X$. Then $\pi$ is the restriction of $\pi_0$ to
$U$ for which it holds:
\[\pi(q)=\arg \min_{p \in X}\|q-p\|.
\]

Then the homotopy is given by
\[F(q,t)=(1-t)q+t\pi(q).\]

It is also important to observe that the retraction $\pi$ moves
the points of $U$ less than $\delta$; this is because the
trajectory of $\pi(q)$ always remains inside a ball of
$U$ that contains $q$ ($q$ can be contained in the
intersection of different balls), for every $q\in U$.
In fact $\pi^{-1}(q)=U \cap Tan_q^{\bot} \cap B_{\tau}(q)$
(for a complete argument we refer to Section 4 of~\cite{NiSmWe08}).
\end{remark}

\section{Retracts}

Aim of this Section is to yield two rather general results, which
will be specialized to Theorem \ref{wrap} and Lemma \ref{t4}.
Throughout this Section, $Y$ will be a compact Riemannian (hence triangulable)
submanifold of $\mathbb R^m$ and $V$ will be a compact, triangulable subspace of
$\mathbb R^m$ such that $Y$ is a deformation retract of $V$, with
retraction $r$ and homotopy $G:V\times I \to V$ from the identity of $V$, $1_V$, to $r$.
Moreover $\forall y \in Y$, $\forall v \in
r^{-1}(y)$, $\forall t\in I$ we assume that $(r\circ G)(v,t)=y$.\\
Let also $ f:\mathbb R^m \to \mathbb R^n$ be a
continuous function, and $ f_Y$ and $ f_V$ be the
restrictions of $ f$ to $Y$ and $V$ respectively.

\begin{lemma}\label{lemmas1}
$Y \langle  f_Y \m  x \rangle$ is a deformation retract of
$V \langle  f_Y \circ r \m  x \rangle$.
\end{lemma}

\begin{proof}
Let $r_{ x}:V \langle  f_Y \circ r \m  x
\rangle\to Y \langle  f_Y \m  x \rangle$ be
the restriction of $r$ to $V \langle  f_Y \circ r \m  x
\rangle$. It is well-defined since, by the definition of the two sets, $r_x(V
\langle  f_Y \circ r \m  x \rangle)\subseteq Y \langle
 f_Y \m  x \rangle$. We now set $G_x:V \langle
f_Y \circ r \m  x \rangle\times I\to V \langle
 f_Y \circ r \m  x \rangle$ as the restriction of $G$ to
$V \langle  f_Y \circ r \m  x \rangle\times I$. This
restriction is well-defined, because the path from $v$ to $r_x(v)$ is all contained in $V \langle  f_Y \circ r \m  x
\rangle$, thanks to the assumptions on $G$ and $r$. Moreover, it is
continuous and for every $v \in V \langle  f_Y \circ r \m  x \rangle$,
$G_x(v,0)=v$ and $G_x(v,1)=r_x(v)$. So it is the searched
for deformation retraction.
\end{proof}

\begin{remark}
Since the homotopy $G$ is relative to $Y$ (i.e. keeps the points
of $Y$ fixed throughout), this is what is called a {\it strong}
deformation retract in~\cite{EiSt}.
\end{remark}

Now let $\eps=\max_{v\in V}\|r(v)-v\|$ and
$ \omega(\eps)=(\Omega(\eps),\ldots,\Omega(\eps))
\in \mathbb{R}^n$, where $\Omega$ is the modulus of continuity of
$ f$ (Def. \ref{modulo}). For sake of simplicity, set now
$\beta_V = \beta_{(V, f_V,i)}$, $\beta_Y = \beta_{(Y, f_Y,i)}$,  $V_x=V\langle{f_V\m x}\rangle$, and
$Y_x=Y\langle{f_Y\m x}\rangle$.

\begin{lemma}\label{lemmas3}
If $( u,  v)$ is a point of $\Delta^+$ and if $ u
+ \omega(\eps) \prec  v - \omega(\eps)$, then

\[\beta_V( u- \omega(\eps),  v+
\omega(\eps))\leq \beta_Y( u, v)\leq
\beta_V( u+ \omega(\eps), v-
\omega(\eps))\]
\end{lemma}

\begin{proof}

For $\omega=\omega(\eps)$ such that $u
+\omega \prec v - \omega$ we can  consider the following diagram
\begin{eqnarray}\label{diagram}
\xymatrix {
V_{u-\omega}\ar[rrrr]^{\alpha} \ar[dddrr]_{\bar r }&  & & &V_{u+\omega}\ar[rr]^{ \gamma} & &V_{v-\omega} \ar[rrrr]^{\delta}\ar[dddrr]_{\hat r} & & & &V_{v+\omega}\\
\\ \\
& & Y_u\ar[uuurr]_{\theta }\ar[rrrrrr]^{\eta}& &  & & & & Y_v\ar[uuurr]_{\iota}& & &
}
\end{eqnarray}
where the maps $\alpha$, $\gamma$, $\delta$, $\eta$, $\theta$, $\iota$ are inclusions,
$\bar{r}$ is the composition of $r_{ u}$, where $r_{u}$ is as in the proof of Lemma
\ref{lemmas1} with ${x}={ u}$, with the inclusion  $i_u$ of $V_{u -  \omega}$ into
$V \langle f_Y \circ r \m  u \rangle$, and $\hat{r}$ is similarly defined as the
composition of $r_{ v}$ with the inclusion $i_v$ of $V_{v -  \omega}$ into
$V \langle f_Y \circ r \m  v \rangle$. Let us observe that, passing to homology and
using the symbol $^*$ to denote the induced homomorphisms, the dimension of the image
of $\delta_*\circ\gamma_*\circ\alpha_*$ is precisely $\beta_{V}( u- \omega,  v+ \omega)$,
that of $\eta_*$ is  $\beta_{Y}(u,v)$, and that of $\gamma_*$ is $\beta_{V}( u+ \omega,  v- \omega)$.

The first claimed inequality will follow from the commutativity of the large trapezoid
in diagram (\ref{diagram}) up to homotopy.  Indeed, if $\delta\circ \gamma\circ\alpha$ is
homotopic to $\iota\circ \eta\circ \bar r$, then, passing to homology, the map
$\delta_*\circ\gamma_*\circ\alpha_*$ is equal to the map $\iota_*\circ \eta_*\circ \bar r_*$.
As a consequence the dimension of the image of $\delta_*\circ\gamma_*\circ\alpha_*$  is not
grater than that of $\eta_*$.
Moreover, the second claimed inequality will follow from the commutativity of the small
trapezoid because, in this case, passing to homology we have $\eta_*= \hat r_*\circ \gamma_*\circ \theta_*$.
Therefore, the dimension of the image of $\eta_*$ is not greater than that of $\gamma_*$.

We begin proving that the small trapezoid commutes exactly. We observe that $\hat{r}$
is the identity map on the points of $Y$. Since
\hbox{$Y_u  \subseteq Y$}, we have that
$\hat{r}\circ\gamma\circ\theta$ is the canonical inclusion of $Y_u $ in $Y_v$.

To prove the commutativity of the large trapezoid up to homotopy, using the exact commutativity
of the small trapezoid, it is sufficient to prove that the two triangles in diagram (\ref{diagram})
commute up to homotopy. As for the left triangle, consider $\bar{G}:V_{u - \omega}\times I \to V_{u +  \omega}$
be the composition $\bar{G} =G_{u} \circ (i_u\times 1_I)$,
where $G_{\vec u}$ is as in the proof of Lemma \ref{lemmas1}.
Now, for every $v \in V_{ u -\omega}$, we have $\bar{G}(v,0)=G(v,0)=v=\alpha(v)$
and $\bar{G}(v,1)=G(v,1)=r(v)=\bar{r}(v)=\theta\circ \bar{r}(v)$. Hence $\alpha$ is homotopic to
$\theta\circ \bar r$. The proof that $\delta$ is homotopic to $ \hat r\circ \iota$ is analogous.

\end{proof}

\begin{remark}
The diagram of the previous proof, setting $v=u+\omega$, shows that the persistence
modules of $Y$ and $V$ are $\omega$-interleaved~\cite{Les15} (strongly $\omega$-interleaved, in the terminology
of~\cite{CaCo*09}), so that their interleaving distance is $\ge \omega$.
\end{remark}

\begin{lemma}\label{carab}
If $( u,  v)$ is a point of $\Delta^+$ such that $ u
+ \omega(\eps) \prec  v - \omega(\eps)$ and
\[\beta_V( u- \omega(\eps),  v+
\omega(\eps)) =
\beta_V( u+ \omega(\eps), v-
\omega(\eps))\]
then
\[\beta_Y( u, v) = \beta_V( u, v)\]
\end{lemma}

\begin{proof}
Straightforward from Lemma \ref{lemmas3} and from the fact that
$\beta_V$ is non-decreasing in the first variable and non-increasing in the second one.
\end{proof}

The next Lemma shows that there may be whole regions (hyperparallelepipeds) where the PBN's of $Y$ and $V$ are
known to agree.

\begin{lemma}\label{carab2}
Let $u, \overline{u}, u', v, \overline{v}, v' \in \mathbb{R}^n $ be such that
$ u\preceq \overline{u}- \omega(\eps), \ \overline{u} + \omega(\eps) \preceq u' \prec v \preceq \overline{v} - \omega(\eps),
\ \overline{v} + \omega(\eps) \preceq v'$. If
\[\beta_V( u, v') = \beta_V(u', v )\]
then
\[\beta_Y( \overline{u}, \overline{v}) = \beta_V( \overline{u}, \overline{v})\]
\end{lemma}

\begin{proof}
Since $\beta_V$ is non-decreasing in the first variable and non-increasing in the second one, we have, with $\omega = \omega(\eps)$,
\[\beta_V(u, v') \le \beta_V(u, \overline{v}+\omega) \le \beta_V(\overline{u}-\omega, \overline{v}+\omega)
\le\]
\[\le \beta_V(\overline{u}+\omega, \overline{v}-\omega) \le \beta_V(\overline{u}+\omega, v) \le \beta_V(u', v) \]
whence, by the hypothesis $\beta_V(u, v') = \beta_V(u', v)$, we get $\beta_V(\overline{u}-\omega, \overline{v}+\omega)
= \beta_V(\overline{u}+\omega, \overline{v}-\omega)$. Lemma \ref{carab} then yields the thesis.
\end{proof}

As a consequence of the preceding lemmas, the regions of $\Delta^+$ where $\beta_V$ and $\beta_Y$ may possibly disagree
can be precisely localized in a neighborhood of the discontinuity set of $\beta_V$ (viewed as an integer function). In the case
$n=1$ where the discontinuity sets are proved to be (possibly infinite) line segments~\cite{FrLa01}, these regions are called
{\em blind strips}. We stress that the position of the blind stips is well-known, since it is determined by the
position of the discontinuity lines of the PBNs of $V$.

\begin{lemma}\label{blind}
If $\overline{u}, \overline{v} \in \mathbb{R}^n$ are such that
\[\beta_Y( \overline{u}, \overline{v}) \neq \beta_V( \overline{u}, \overline{v})\]
then there is at least a point $(\tilde{u}, \tilde{v})$ with max norm
$\| (\overline{u}, \overline{v}) - (\tilde{u}, \tilde{v}) \| \le \Omega(\eps)$, which is either an element of the boundary of $\Delta^+$ or a discontinuity
point of $\beta_V$.
\end{lemma}

\begin{proof}
We prove the contrapositive.

Let $(\overline{u}, \overline{v}) \in \mathbb{R}^n$ be a point of $\Delta^+$.
Let $D$ be the closed ball of radius $\Omega$ centered at $(\overline{u}, \overline{v})$.
Set $\omega=\omega(\eps)$, $u=\overline{u}-\omega$, $u'=\overline{u}+\omega$, $v=\overline{v}-\omega$,
$v'=\overline{v}+\omega$; $D$ is the closed hypercube
$\{ (\hat{u}, \hat{v}) \in \mathbb{R}^n \,\vert\, (u, v) \preceq (\hat{u}, \hat{v}) \preceq (u', v') \}$
and its boundary contains $(u, v')$ and $(u', v)$.

Now, assume that $(\overline{u}, \overline{v})$ is at distance
greater than $\Omega$ from the boundary of $\Delta^+$ (with the max norm distance). Then the point
$(u', v)$, belonging to $D$, is contained in $\Delta^+$. This means that $u' \prec v$.

Moreover, assume that $(\overline{u}, \overline{v})$ is at distance
greater than $\Omega$ from any discontinuity point of $\beta_V$.
Then there is an open subset of $\Delta^+$, containing $D$, on whose points $\beta_V$ is continuous.
Since the range of $\beta_V$ has the discrete topology, this implies that $\beta_V$ is constant on this whole
open set, hence also on $D$. Then $\beta_V(u, v')=\beta_V(u', v)$; therefore
$\beta_Y( \overline{u}, \overline{v}) = \beta_V( \overline{u}, \overline{v})$
by Lemma \ref{carab2}.
\end{proof}

%
%

\section{Ball coverings}\label{main}

Throughout this Section, $X$ will be a compact Riemannian (triangulable)
submanifold of $\mathbb R^m$. As hinted in Section \ref{Intro}, we want to get
information on $X$ out of a finite set of points. First, the points will be
sampled on $X$ itself, then even in a (narrow) neighborhood. In both cases,
the idea is to consider a covering of $X$ made of balls centered on
the sampling points.

What we get, is a double inequality which yields
an estimate of the PBNs of $X$ within a fixed distance from
the discontinuity sets of the PBNs (meant as integer functions on $\Delta^+$) of
the union $U$ of the balls of the covering, but Lemma \ref{carab2} even offers the exact
value of it at points sufficiently far from the discontinuity sets.

\subsection{Points on $X$}\label{app}

Let $\delta < \sqrt{\frac{3}{5}}\tau$ and let $L=\{l_1,\ldots,l_k\}$ be a set of points of $X$ such
that for every $p\in X$ there exists an $l_j\in L$ for which
$\|p-l_j\| < \dfrac{\delta}{2}$. Let $U$ be the union of the
balls $B(l_j, \delta)$ of radius $\delta$ centered at $l_j$, $j=1,\ldots, k$. So
all conditions of
Proposition \ref{smale} are satisfied. As before, set $\beta_U = \beta_{(U, f_U,i)}$ and $\beta_X = \beta_{(X, f_X,i)}$. \\

\begin{theorem}\label{wrap}
If $( u,  v)$ is a point of $\Delta^+$ and if $ u + \omega(\delta) \prec
 v - \omega(\delta)$, where $
\omega(\delta)=(\Omega(\delta),\ldots,\Omega(\delta)) \in \mathbb{R}^n$,
then
\[\beta_U( u- \omega(\delta),  v+ \omega(\delta))\leq \beta_X( u, v)\leq \beta_U( u+ \omega(\delta), v- \omega(\delta))\]

If $\overline{u}, \overline{v} \in \mathbb{R}^n$ are such that
\[\beta_X( \overline{u}, \overline{v}) \neq \beta_U( \overline{u}, \overline{v})\]
then there is at least a point $(\tilde{u}, \tilde{v})$ with max norm
$\| (\overline{u}, \overline{v}) - (\tilde{u}, \tilde{v} \| \le \Omega(\delta)$, which is either an element of the boundary of $\Delta^+$ or a discontinuity
point of $\beta_U$.

\begin{proof}
By Lemmas \ref{lemmas3} and \ref{blind}, with $Y=X$, $V=U$.
\end{proof}

\end{theorem}

\subsection{Points near $X$}\label{point}

So far we have approximated $X$ by points picked up on
$X$ itself, but it is also possible to choose the points near $X$,
by respecting some constraints. Once more, this is possible thanks
to a result of~\cite{NiSmWe08}.

\begin{proposition} (Prop. 7.1 of~\cite{NiSmWe08}) \label{smale2}
Let $L=\{l_1,\ldots,l_k\}$ be a set of points in the tubular
neighborhood of radius $s$ around $X$ and
$U=\bigcup_{j=1,\ldots,k}B(l_j,\delta)$ be the union of the balls of
$\mathbb R^m$ centered at the points of $L$ and with radius
$\delta$. If for every point $p\in X$, there exist a point $l_j
\in L$ such that $\|p-l_j\|<s$, then $U$ is a
deformation retract of $X$, for all $s<(\sqrt{9}-\sqrt{8})\tau$
and
$\delta\in\left(\frac{(s+\tau)-\sqrt{s^2+\tau^2-6s\tau}}{2},\frac{(s+\tau)
+ \sqrt{s^2+\tau^2-6s\tau}}{2}\right).$

\end{proposition}

Then, as with Theorem \ref{wrap}, we have, with an analogous proof:

\begin{theorem}\label{wrap2}
Under the hypotheses of Proposition \ref{smale2}, if $( u,  v)$
is a point of $\Delta^+$ and if $ u + \omega(\delta+s) \prec
 v - \omega(\delta+s)$, where $
\omega(\delta+s)=(\Omega(\delta+s),\ldots,\Omega(\delta+s)) \in \mathbb{R}^n$,
then
\[\beta_U( u- \omega(\delta+s),  v+ \omega(\delta+s))\leq \beta_X( u, v)\leq \beta_U( u+ \omega(\delta+s), v- \omega(\delta+s))\]

If $\overline{u}, \overline{v} \in \mathbb{R}^n$ are such that
\[\beta_X( \overline{u}, \overline{v}) \neq \beta_U( \overline{u}, \overline{v})\]
then there is at least a point $(\tilde{u}, \tilde{v})$ with max norm
$\| (\overline{u}, \overline{v}) - (\tilde{u}, \tilde{v} \| \le \Omega(\delta+s)$, which is either an element of the boundary of $\Delta^+$ or a discontinuity
point of $\beta_U$.

\end{theorem}

\subsection{The 1D case}\label{ex}

We now show how Theorem \ref{wrap} can be used for
applications. In the case of filtering functions with one-dimensional range, and in the 1D reduction
of Section 2.1 of~\cite{CaDiFe} and Section 4 of~\cite{CeDiFeFrLa13}, the subsets of $\mathbb{R}^2$
where the PBN functions are not continuous
form (possibly infinite) line segments:~\cite{FrLa01}. We recall that by Lemma \ref{blind},
the {\em blind strips}, i.e. the regions where the equality $\beta_Y( \overline{u}, \overline{v}) = \beta_V( \overline{u}, \overline{v})$
is not granted, are $2 \omega(\eps)$ wide strips around such segments.

In the case of ball coverings as in Section \ref{app}, the width of the blind
strips is a representation of the approximation error, in that it is
directly related to $\Omega(\delta)$, where $1/\delta$ represents the density
of the sampling.

Let $X$ be a circle of radius 4 in $\mathbb R^2$
(Figure \ref{fig1}); we observe that $\tau$ is exactly the radius
of $X$, so $\tau=4$. In order to create a well defined approximation we
need that $\delta < \sqrt{\frac{3}{5}}\tau$.

\begin{figure}[ht]
  \centering
  \begin{minipage}[ht]{0.4\linewidth}
     \centering
     \includegraphics[scale=0.2]{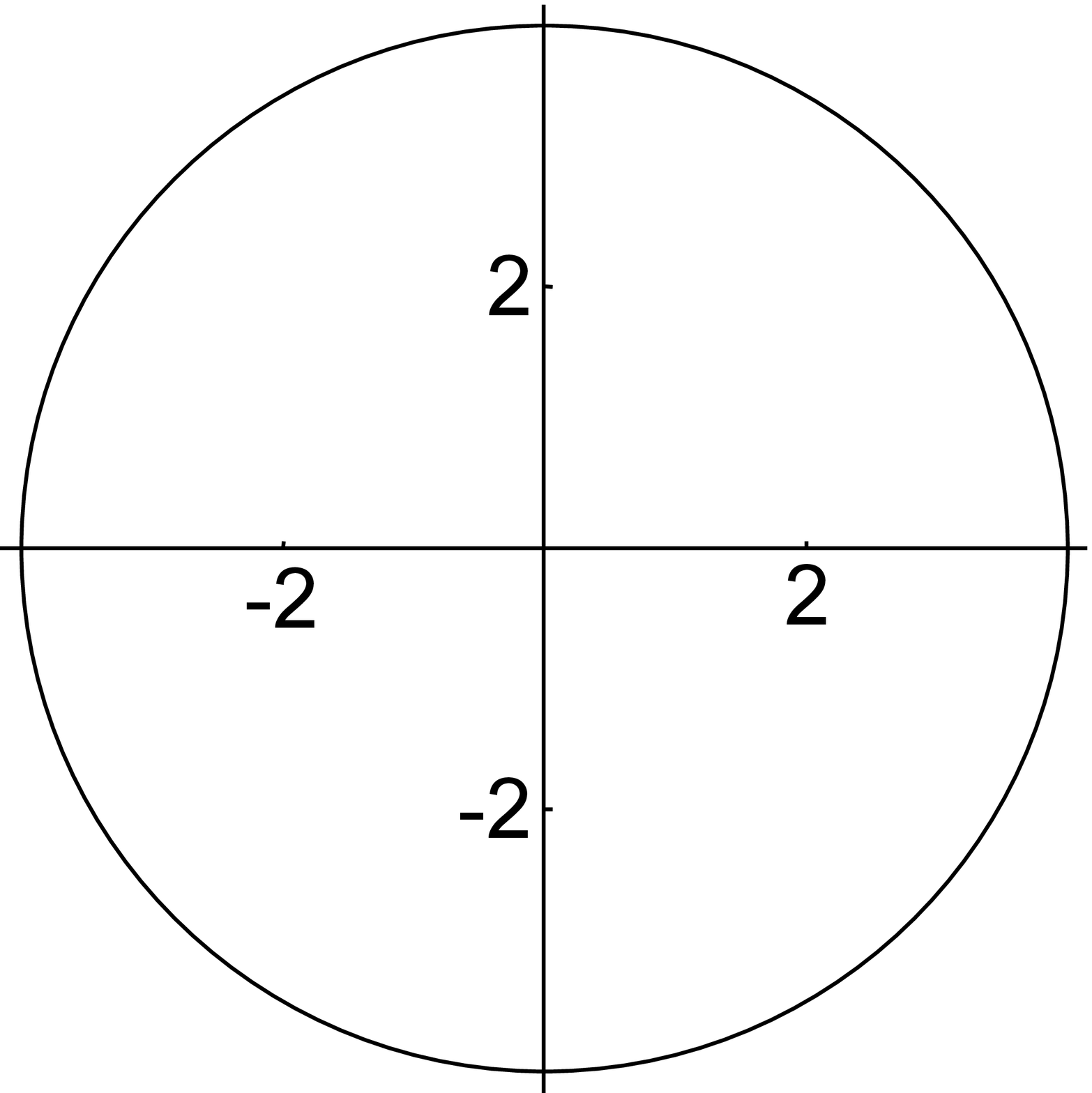}
     \caption{The circle of radius 4, $X$.} \label{fig1}
  \end{minipage}%
  \hfill%
  \begin{minipage}[ht]{0.4\linewidth}
     \centering
     \includegraphics[scale=0.21]{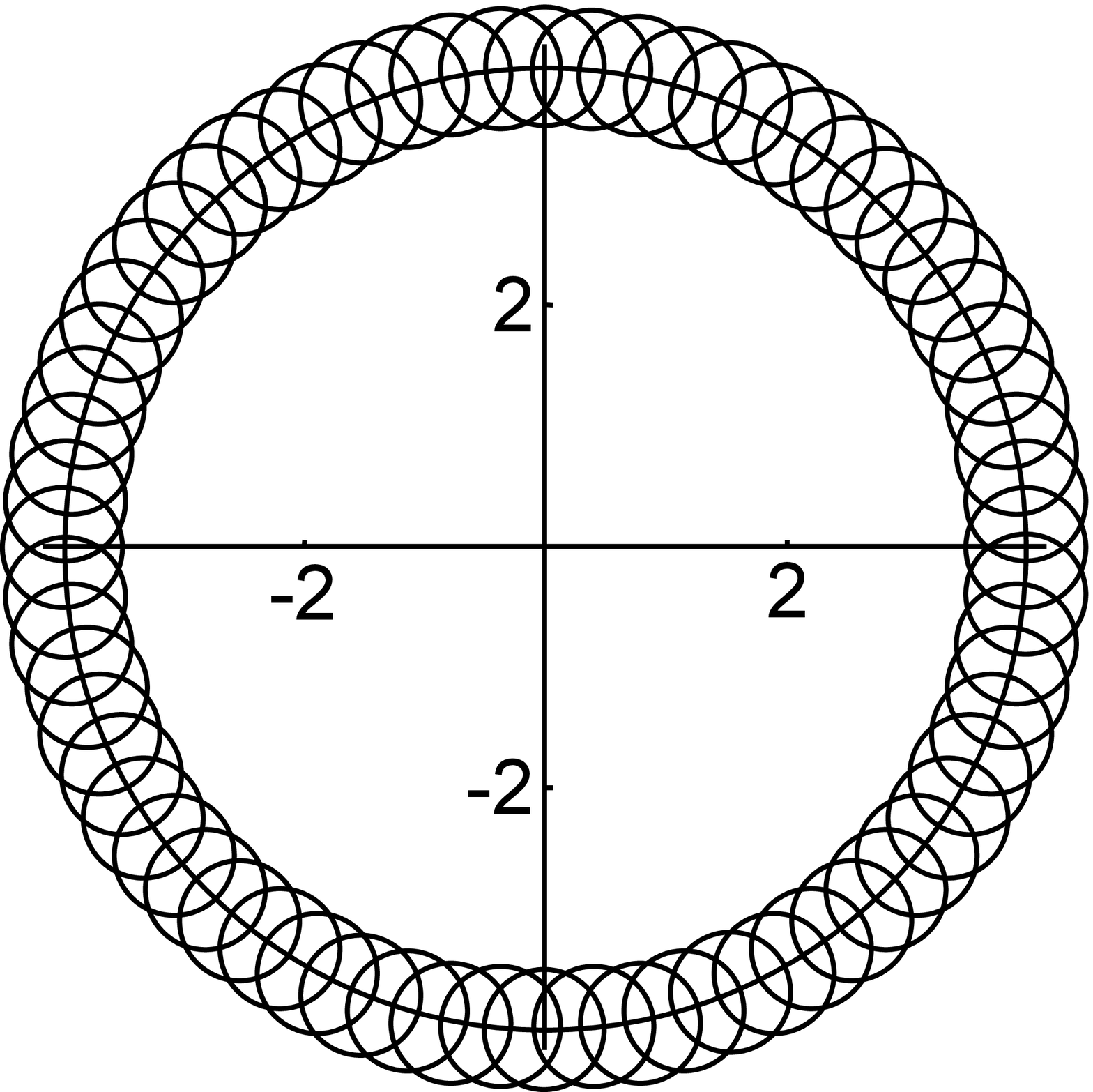}
     \caption{The ball union $U$.} \label{fig2}
  \end{minipage}
\end{figure}

In the first example
we have taken $\delta=0.5$. Now, to satisfy the hypothesis of
Theorem \ref{wrap} (that for every $p\in X$ there exists an $l_j\in L$
such that $\|p-l_j\|<\frac{\delta}{2}$), we have
chosen 64 points $l_j$ on X. Moreover we have sampled $X$
uniformly, so that there is a point every $\frac{\pi}{32}$ radians
(Figure \ref{fig2}). We stick to the monodimensional case, choosing \hbox{$f:\mathbb R^2
\to \mathbb R$}, with \hbox{$f(x,y)=\mathrm{abs}(y)$}. $U$ is the resulting ball union.

\begin{figure}[htbp]
\centering
  \begin{minipage}[ht]{0.48\linewidth}
     \centering
     \includegraphics[scale=0.2]{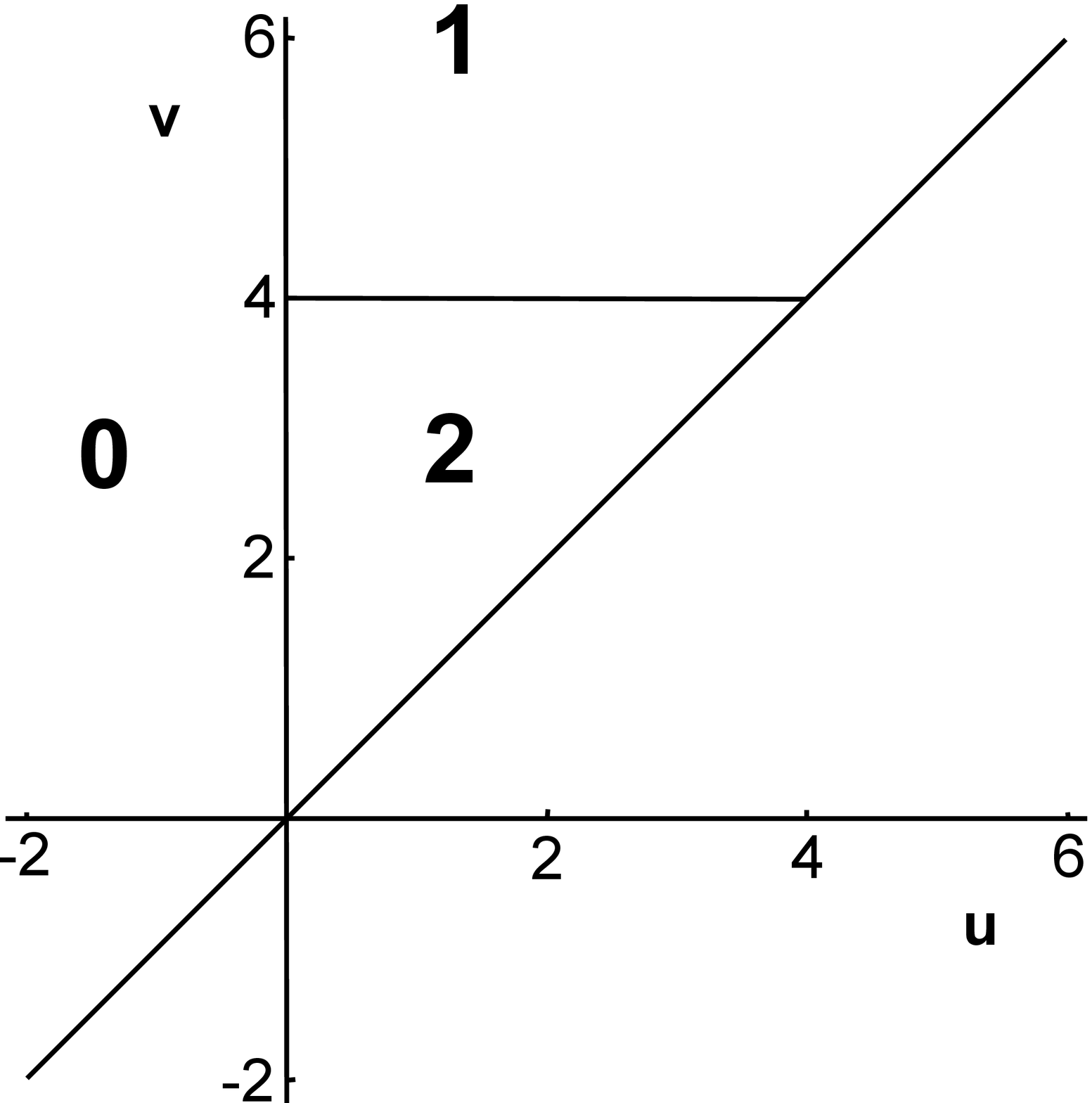}
     \caption{The representation of $\beta_{(X,f_X,0)}$, the $0$-PBNs of $X$.} \label{fig3}
  \end{minipage}
  \hfill%
  \begin{minipage}[ht]{0.51\linewidth}
     \centering
     \includegraphics[scale=0.2]{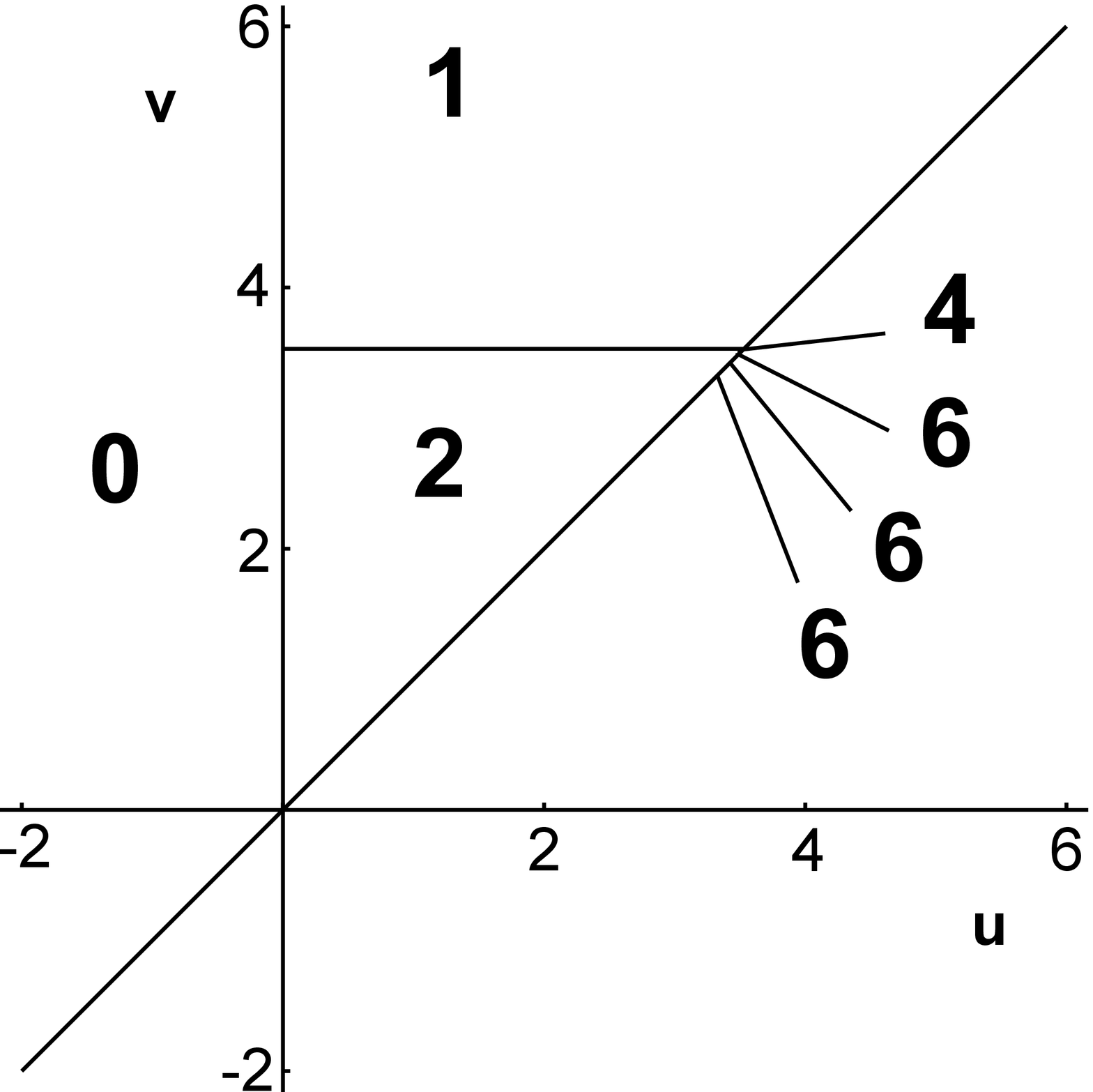}
     \caption{The representation of $\beta_{(U,f_U,0)}$,
     the $0$-PBNs of the ball union $U$.} \label{fig4}
  \end{minipage}
  \hfill%
  \begin{minipage}[ht]{0.5\linewidth}
     \centering
     \includegraphics[scale=0.4]{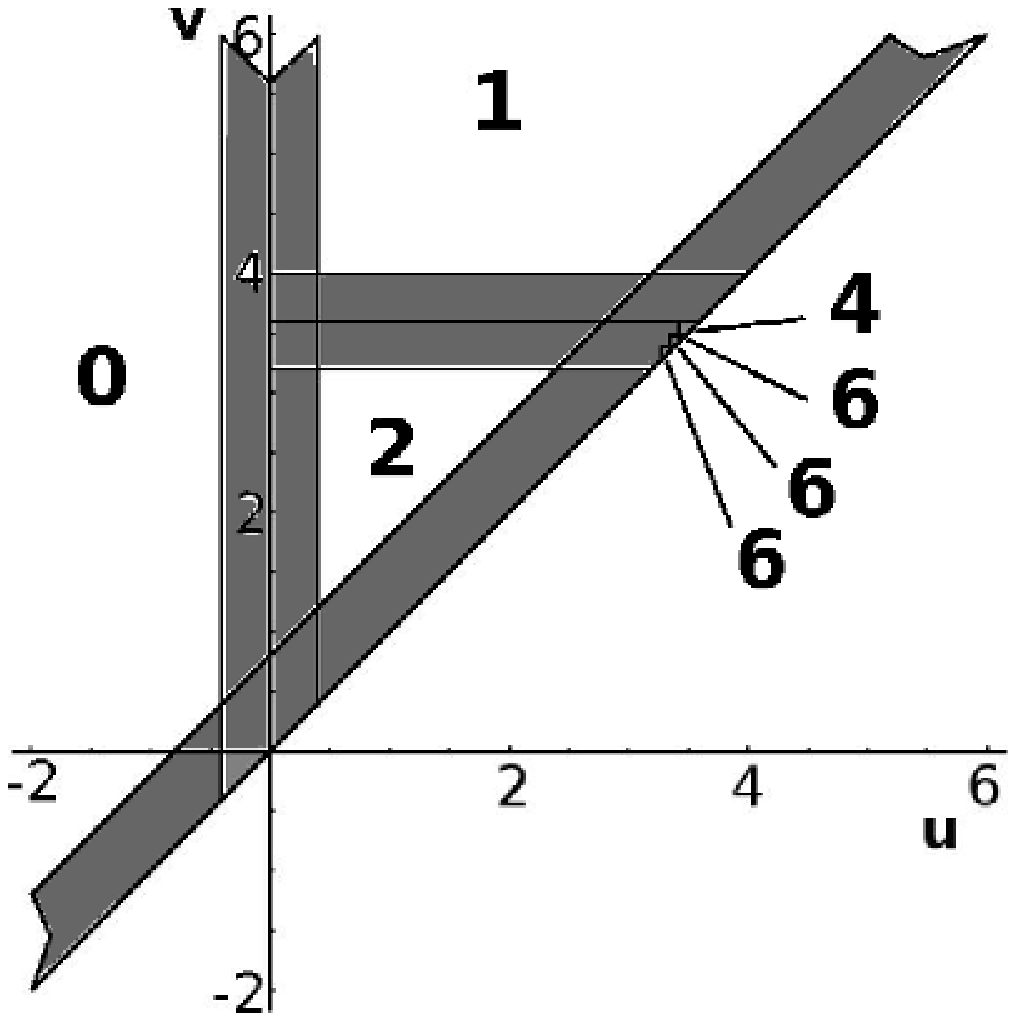}
     \caption{The blind strips of $\beta_{(U,f_U,0)}$.} \label{fig5}
  \end{minipage}
\end{figure}

Figures \ref{fig3} and \ref{fig4} represent the PBN functions at degree zero of $X$ and
$U$ respectively. $\Delta^+$ is the half-plane above the diagonal line, and the numbers
are the values of the PBNs in the triangular regions they are written in.
In Figure \ref{fig3} there is only one big triangle where the value 2 signals the
two different connected components generated by $f_X$. The two
connected components collapse to one at value $4$. In Figure
\ref{fig4} there is also a big triangle representing the two
connected components, but they collapse at value $3.53106$.
Moreover there are 4 other very small triangles near the diagonal,
representing more connected components generated by the approximation
error. In the last figure (Figure \ref{fig5})
the blind strips around the discontinuity lines of
$\beta_{(U,f_U,0)}$ are shown. The width of these strips,
since $\Omega(\delta) = 0.5$, is equal to $2\Omega(\delta)=1$.
This figure illustrates the idea underlying Theorem \ref{wrap}. Taken a
point $(u,v)$ outside the strips, the values of the PBNs of $U$ at $(u-\Omega(\delta),v+\Omega(\delta))$
and $(u+\Omega(\delta),v-\Omega(\delta))$ are the same. So also
the value of the PBNs of $X$ at $(u,v)$ is determined.
Figures \ref{fig3'}, \ref{fig4'}, \ref{fig5'} depict, in analogous way, the
(obviously much simpler) PBNs of degree 1.

\begin{figure}[htbp]
\centering
  \begin{minipage}[ht]{0.48\linewidth}
     \centering
     \includegraphics[scale=0.2]{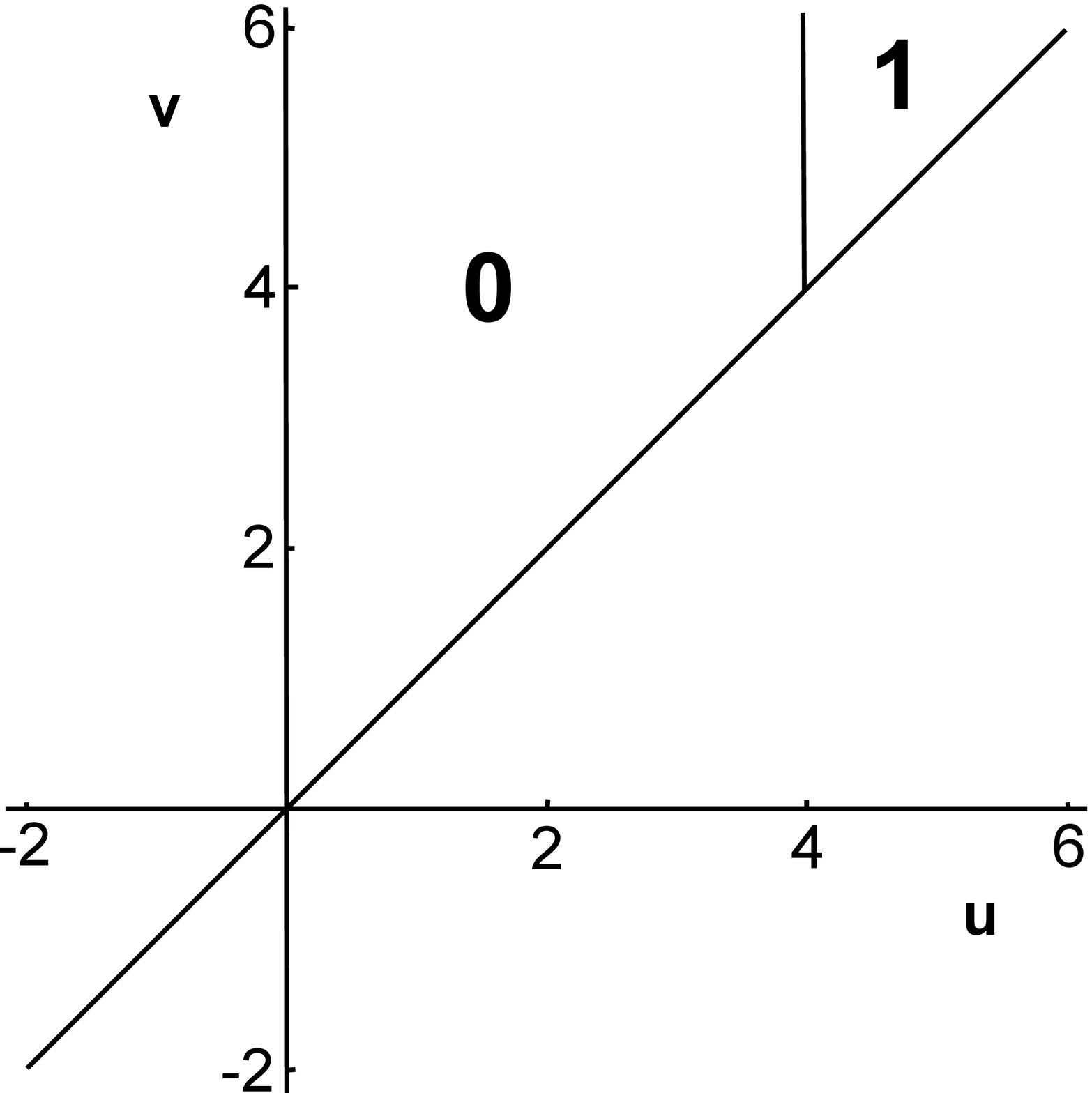}
     \caption{The representation of $\beta_{(X,f_X,1)}$, the $1$-PBNs of $X$.} \label{fig3'}
  \end{minipage}
  \hfill%
  \begin{minipage}[ht]{0.51\linewidth}
     \centering
     \includegraphics[scale=0.2]{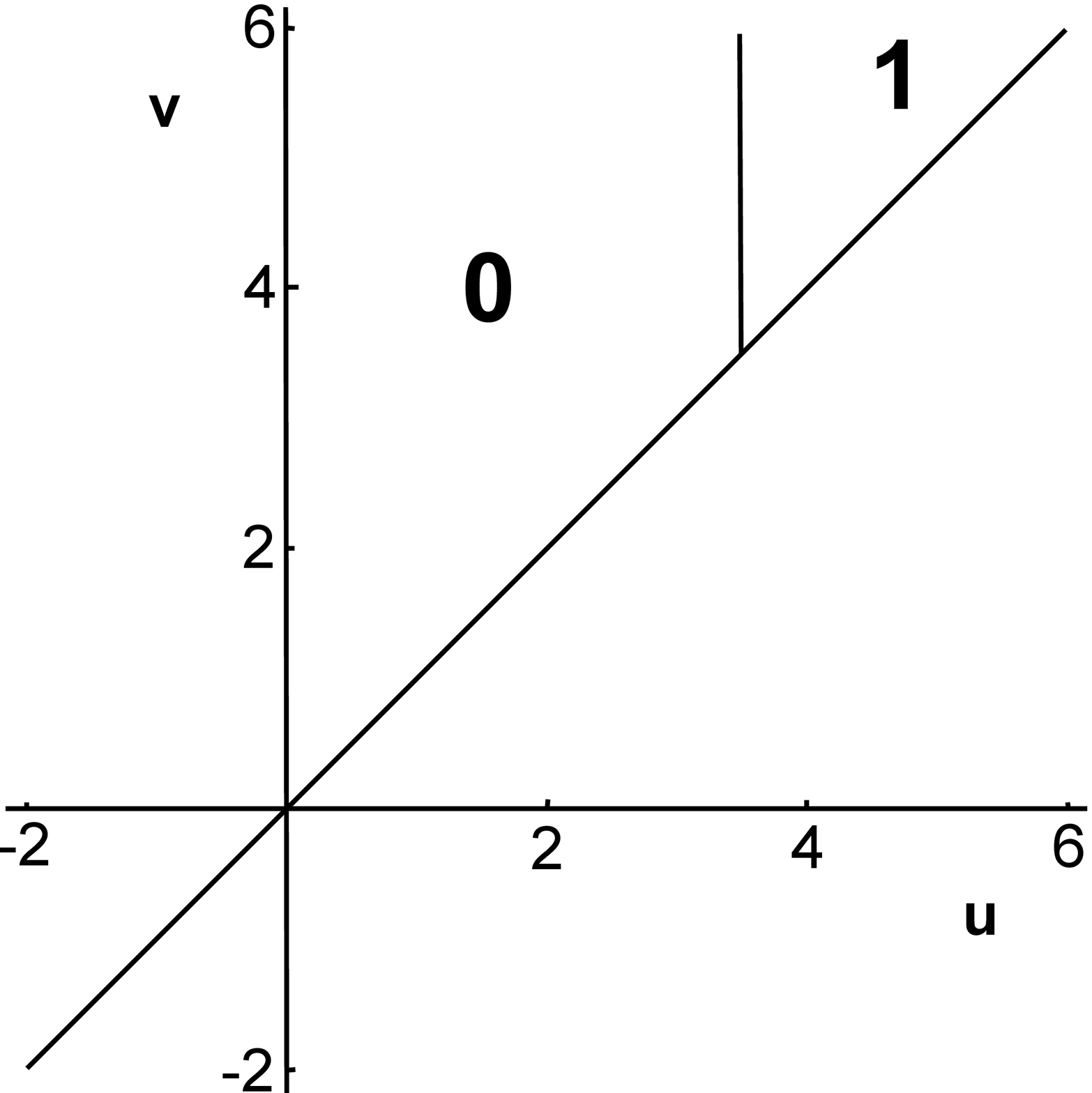}
     \caption{The representation of $\beta_{(U,f_U,1)}$,
     the $1$-PBNs of the ball union $U$.} \label{fig4'}
  \end{minipage}
  \hfill%
  \begin{minipage}[ht]{0.5\linewidth}
     \centering
     \includegraphics[scale=0.4]{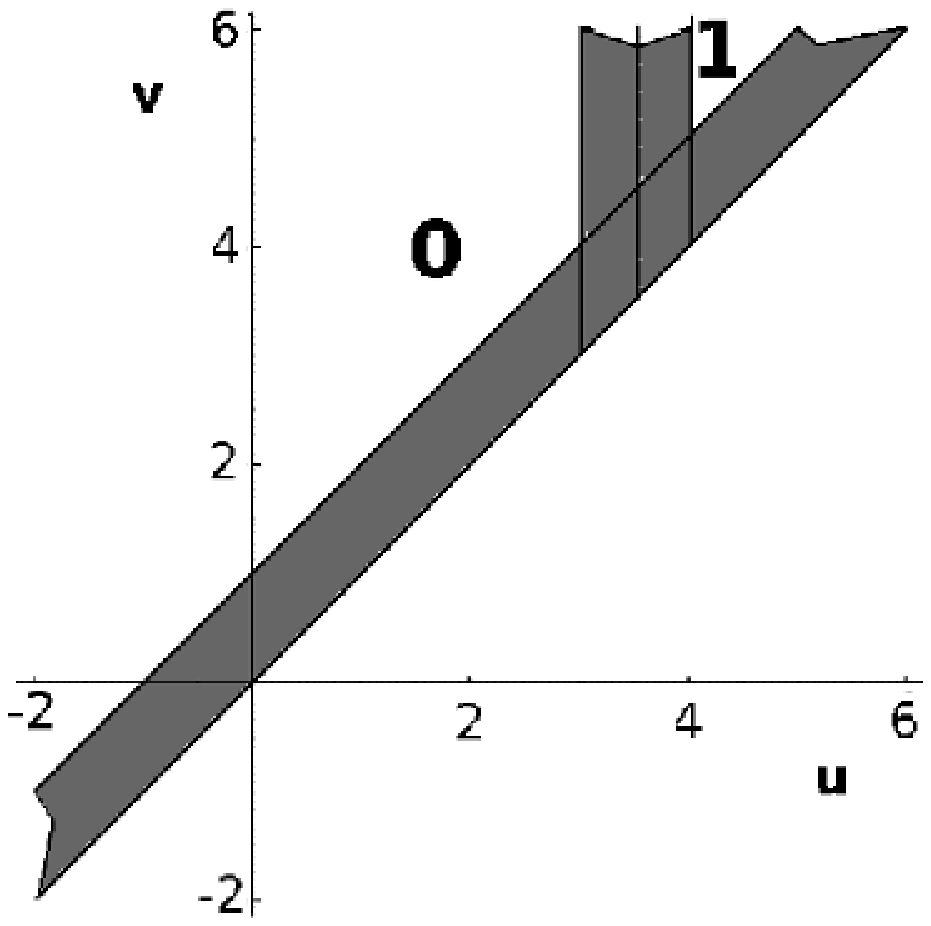}
     \caption{The blind strips of $\beta_{(U,f_U,1)}$.} \label{fig5'}
  \end{minipage}
\end{figure}

For a second example we have chosen the points $l_j$ not
necessarily on $X$. We have satisfied the hypothesis of Proposition
\ref{smale2}, choosing $s=0.25$ and $\delta=0.55$. Then, in order to cover $X$ well, we have chosen a point
every $\frac{\pi}{48}$ radians, for a total of 96 points. But this
time the points are either $0$ or $0.1$ or $0.2$ away from $X$. Figure
\ref{fig6} shows the resulting ball union $U'$. As in
the previous case, in the representation of
$\beta_{(U',f_{U'},0)}$ (Figure \ref{fig7})
there is one big triangle showing two connected components and this
time they collapse at value $3.40955$. Compared to Figure
\ref{fig4}, there are many more small triangles generated by the
asymmetry of the sampling. The width of the blind strips in
Figure \ref{fig8} is $2\Omega(\delta+s)=1.6$, so there is still
the central triangle. This means that, although the error in the
approximation is much bigger, the blind strips do not cover the entire
figure, leaving the topological information intact at least in some
small areas of $\Delta^+$.

\begin{figure}[htbp]
\centering
  \begin{minipage}[ht]{0.45\linewidth}
     \centering
     \includegraphics[scale=0.2]{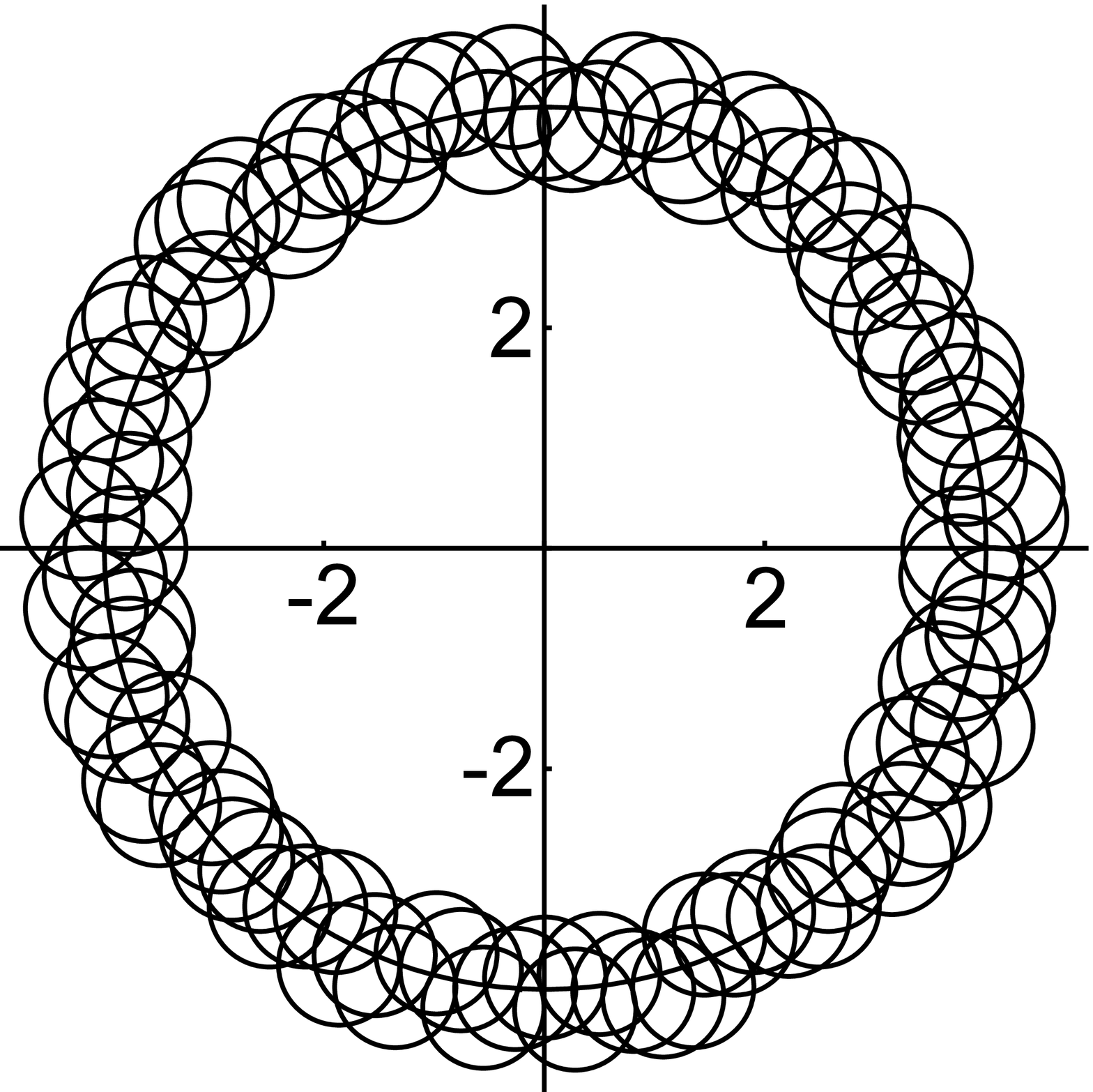}
     \caption{The ball union $U'$.} \label{fig6}
  \end{minipage}
  \hfill%
  \begin{minipage}[ht]{0.51\linewidth}
     \centering
     \includegraphics[scale=0.2]{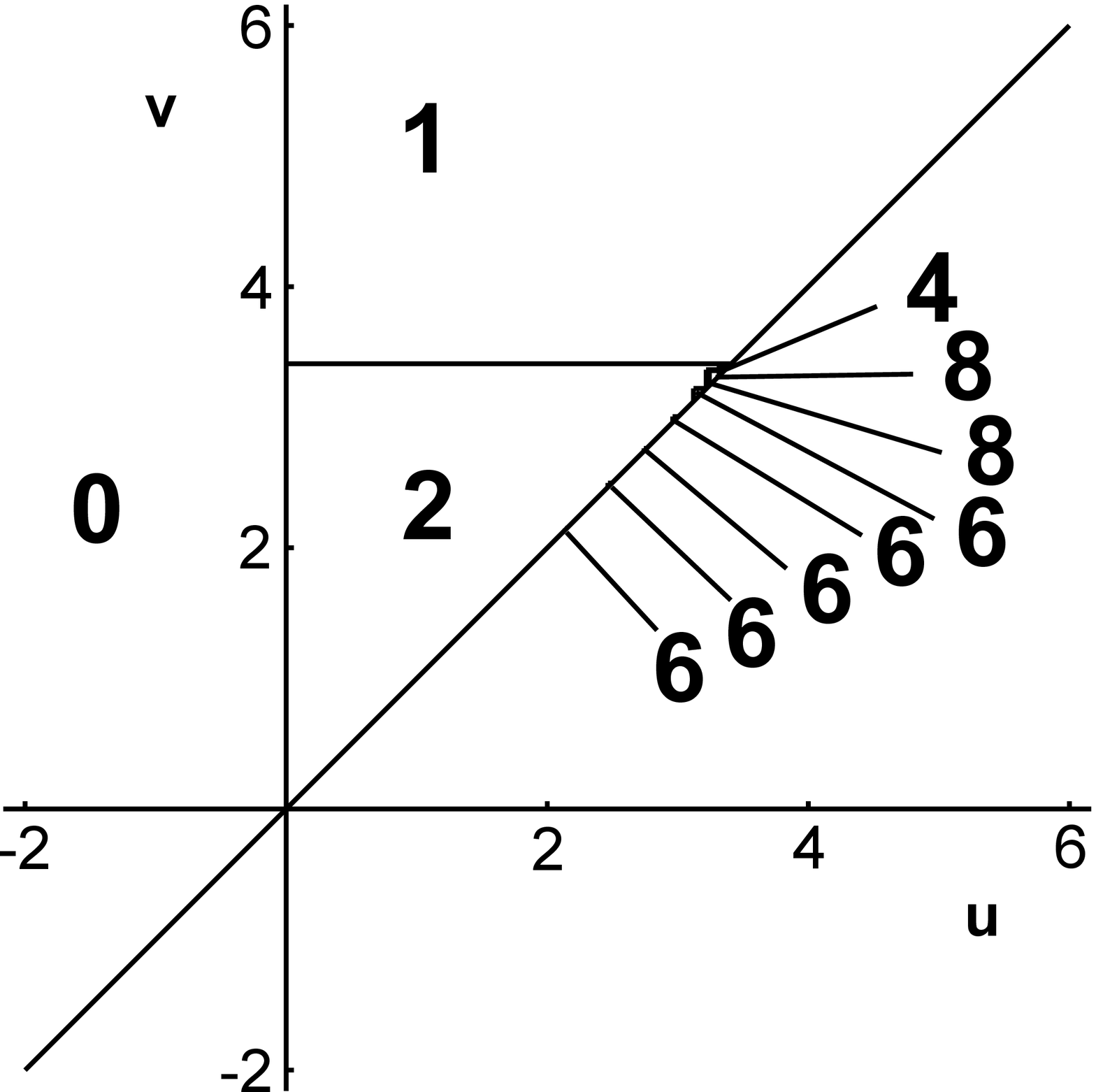}
     \caption{The representation of $\beta_{(U',f_{U'},0)}$,
     the $0$-PBNs of the ball union $U'$.} \label{fig7}
  \end{minipage}
  \hfill%
  \begin{minipage}[ht]{0.5\linewidth}
     \centering
     \includegraphics[scale=0.5]{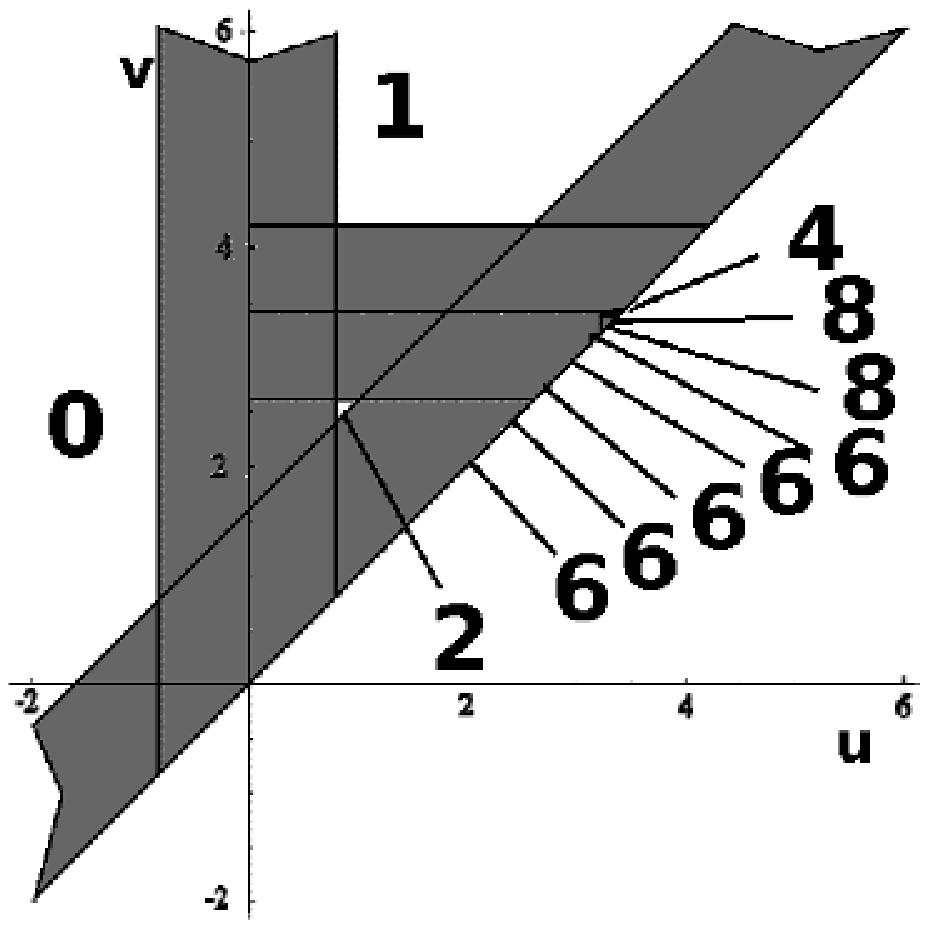}
     \caption{The blind strips of $\beta_{(U',f_{U'},0)}$.} \label{fig8}
  \end{minipage}
\end{figure}

\section{Shape comparison}\label{comparison}

So, outside the blind strips, we can get the values of PBNs of the sampled object $X$ out of
its ball covering $U$ (as always, given a filtering function $ f$ defined on the ambient space $\mathbb R^m$).
This fact gives us the possibility of assessing shape dissimilarity of two pairs $(X,  f)$, $(Y,  g)$
although we only know ball coverings $U$, $W$ of them, respectively. For this we need the notion of natural pseudodistance,
introduced in~\cite{DoFr04} and further studied in~\cite{DoFr07,DoFr09}, and its relationship with PBNs.

We recall that, for any two topological spaces $X,Y$ endowed with
two continuous functions $ f: X \to \R^n$, $ g: Y
\to \R^n$, we have the following definition.

\begin{definition}
The \emph{natural pseudodistance} between the pairs $(X,
f)$ and $(Y, g)$, denoted by
$\delta\left((X,  f),(Y,  g)\right)$,  is
\begin{itemize}
\item[(i)] the number $\inf_{h}\max _{x\in X}\|
f(x)- g(h(x))\|$ where $h$ varies in the set
$H(X,Y)$ of all the homeomorphisms between $X$ and $Y$, if $X$ and
$Y$ are homeomorphic;
\item[(ii)] $+\infty$, if $X$ and $Y$ are
not homeomorphic.
\end{itemize}
\end{definition}

\subsection{Lower bounds for the natural pseudodistance}

The following theorem, which extends Theorem 1 of~\cite{DoFr04}, provides us with a lower
bound for the natural pseudodistance.

\begin{theorem}\label{bound}
Let $(X,  f)$ and $(Y,  g)$ be two pairs. If, for some degree $i$,
$\beta_{(X, f,i)}( u,  v) >$ $\beta_{(Y, g,i)}(u', v')$
then
\[\delta\left((X,  f), (Y,  g)\right) \ge \min\left\{ \min_r\{u'_r-u_r\}, \min_r\{v_r-v'_r\}\right\}\]
\end{theorem}

\begin{proof}
Set $\xi = \min\left\{ \min_r\{u'_r-u_r\}, \min_r\{v_r-v'_r\}\right\}$. If $\xi \le 0$ then the thesis trivially holds;
so assume $\xi >0$.
Since $\beta_{(X, f,i)}( u,  v) > 0$ we have $X \langle  f \m  u \rangle \ne \emptyset$.
Assume that the thesis does not hold, so
$\delta\left((X,  f), (Y,  g)\right) < \xi$. Under this assumption
there exists a homeomorphism $h:X \to Y$ such that
\[\max _{x\in X}\| f(x)- g(h(x))\|< \xi.\]
Set $\Theta(h) = \max _{x\in X}\| f(x)- g(h(x))\|$.
So $\Theta(h) < \xi$, and, setting $\vec{\zeta}=(\xi, \dots, \xi)$ we have that $h$ maps the lower level subset
$X \langle  f \m  u \rangle$ into  $Y \langle  g \m  u + \zeta \rangle \subset Y \langle  g \m u' \rangle$.
Analogously, $h^{-1}$ maps $Y \langle  g \m  v' \rangle$ into $X \langle  f \m v' + \zeta \rangle \subset X \langle  f \m  v \rangle$.
Set $X_u = X \langle  f \m  u \rangle$, $X_{u'} = h^{-1}(Y \langle  g \m u' \rangle)$,
$X_{v'} = h^{-1}(Y \langle  g \m  v' \rangle)$, $X_v = X \langle  f \m  v \rangle$. Then
there are inclusion maps

\[\iota_{ u, u'}: X_u \hookrightarrow X_{u'}, \; \; \; \;
\iota_{u', v'}: X_{u'} \hookrightarrow X_{v'}, \; \; \; \; \iota_{v', v}: X_{v'} \hookrightarrow X_v\]

\noindent and
$\iota_{u, v}: X_u \hookrightarrow X_v$, with $\iota_{u, v} =
\iota_{v', v} \circ \iota_{u', v'} \circ \iota_{u, u'}$.

The pairs $\big( Y \langle  g \m v' \rangle, Y \langle  g \m u'\rangle \big)$ and $(X_{v'}, X_{u'})$ are homeomorphic,
so they can be interchanged in what follows.
Our claim is that for each degree $i$ the dimension of the image of the homomorphism induced in homology by $\iota_{u, v}$
is less than or equal to the one of the image of the homomorphism induced by $\iota_{u', v'}$.
First, note that all dimensions considered here are finite (Theorem 2.3 of~\cite{CeDiFeFrLa13}). Then, each inclusion map $\iota$ induces a homology homomorphism
$\iota_*$, and $(\iota_{u, v})_* = (\iota_{v', v})_* \circ (\iota_{u', v'})_* \circ (\iota_{u, u'})_*$. But then
dimIm$(\iota_{u, v})_* \le$ dimIm$(\iota_{u', v'})_*$, so $\beta_{(X, f,i)}( u,  v) \le$
dimIm$(\iota_{u', v'})_* = \beta_{(Y, g,i)}(u', v')$ against the hypothesis.
\end{proof}

\subsection{An example of comparison of sampled shapes}

When only finite, dense enough samples - or, equivalently, ball coverings - of two objects are available,
if there is a nonempty intersection of the complements of the blind strips, we can still assess the
natural pseudodistance between them.

This is the case of the following example: We obtain a lower bound for $\delta(X, Y)$, where $X$ is the circle
and $Y$ is the bean-shaped curve of Figure \ref{circlecomp} and Figure \ref{beancomp}.

\begin{figure}[htbp]
  \centering
    \begin{minipage}[ht]{0.5\linewidth}
     \centering
     \includegraphics[scale=0.2]{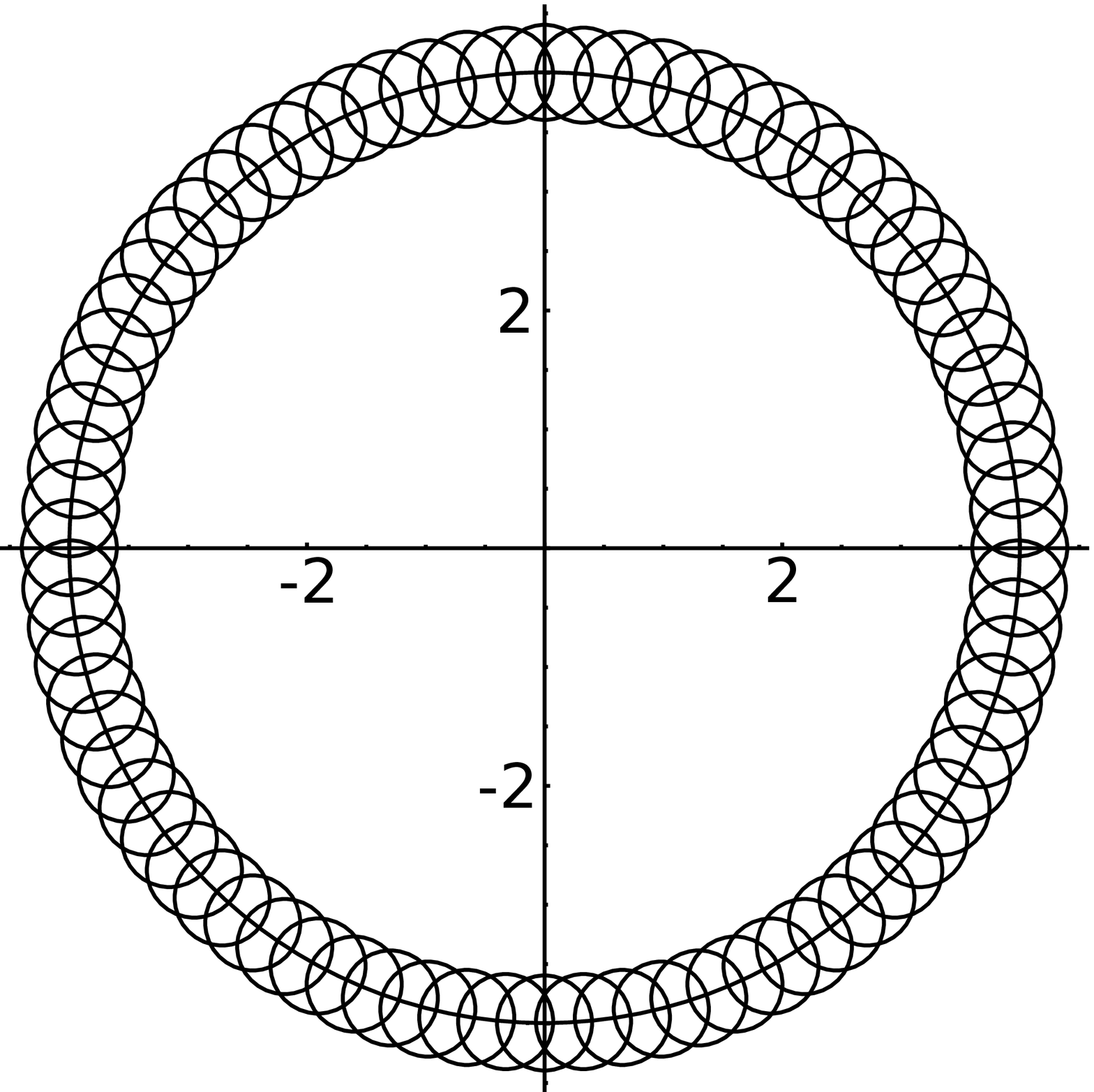}
     \caption{A circle $X$ covered with balls of radius 0.4.} \label{circlecomp}
  \end{minipage}%
  \hfill%
  \begin{minipage}[ht]{0.5\linewidth}
     \centering
     \includegraphics[scale=0.25]{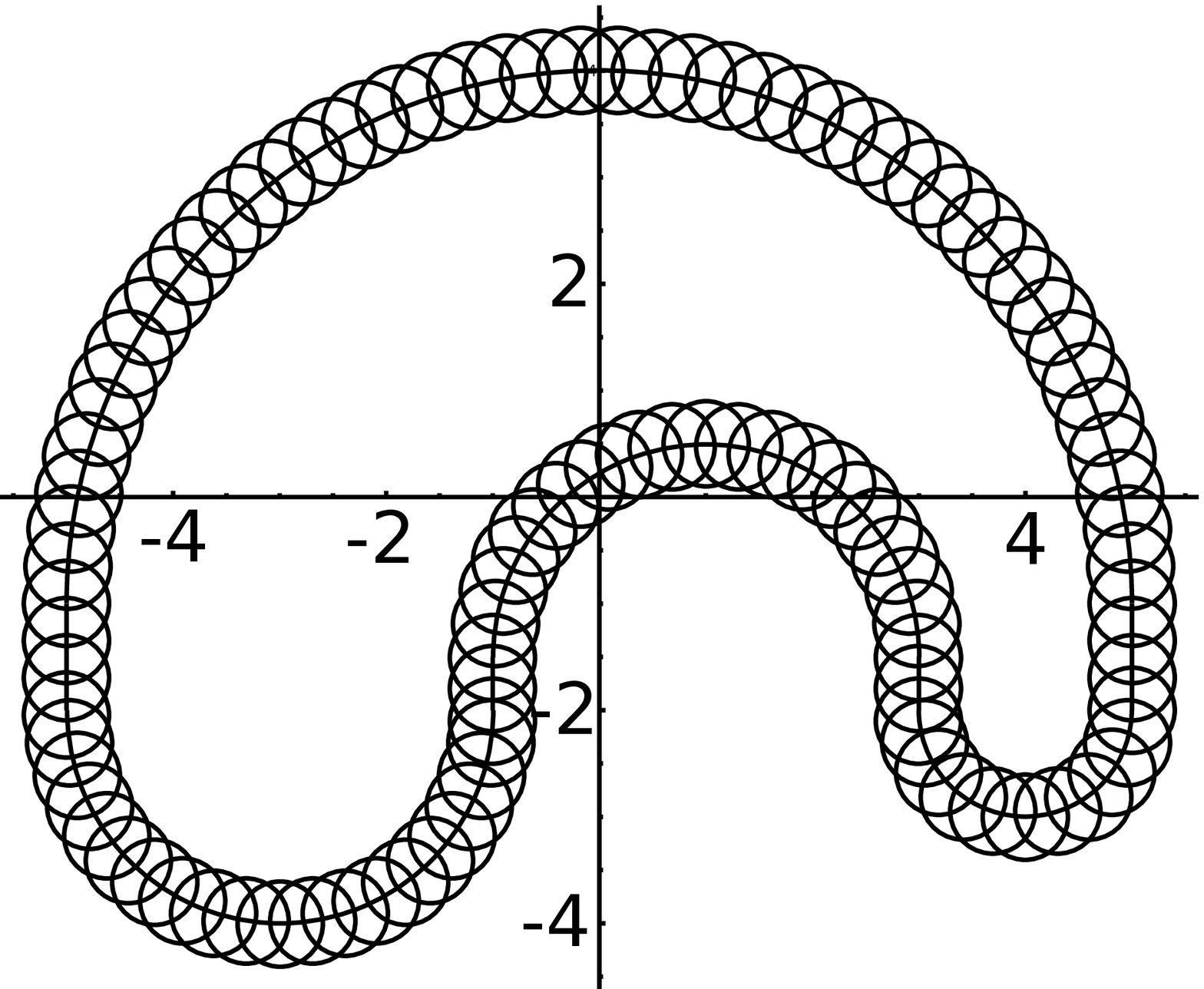}
     \caption{A different curve $Y$ covered with balls of radius 0.4.} \label{beancomp}
  \end{minipage}
\end{figure}

For both spaces the filtering functions $f$ and $g$ are the restrictions of the absolute value of ordinate. The 0-PBNs for $(X,f)$
and $(Y,g)$ are depicted in Figures \ref{circleh0} and \ref{beanh0} respectively.

\begin{figure}[htbp]
  \centering
    \begin{minipage}[ht]{0.5\linewidth}
     \centering
     \includegraphics[scale=0.2]{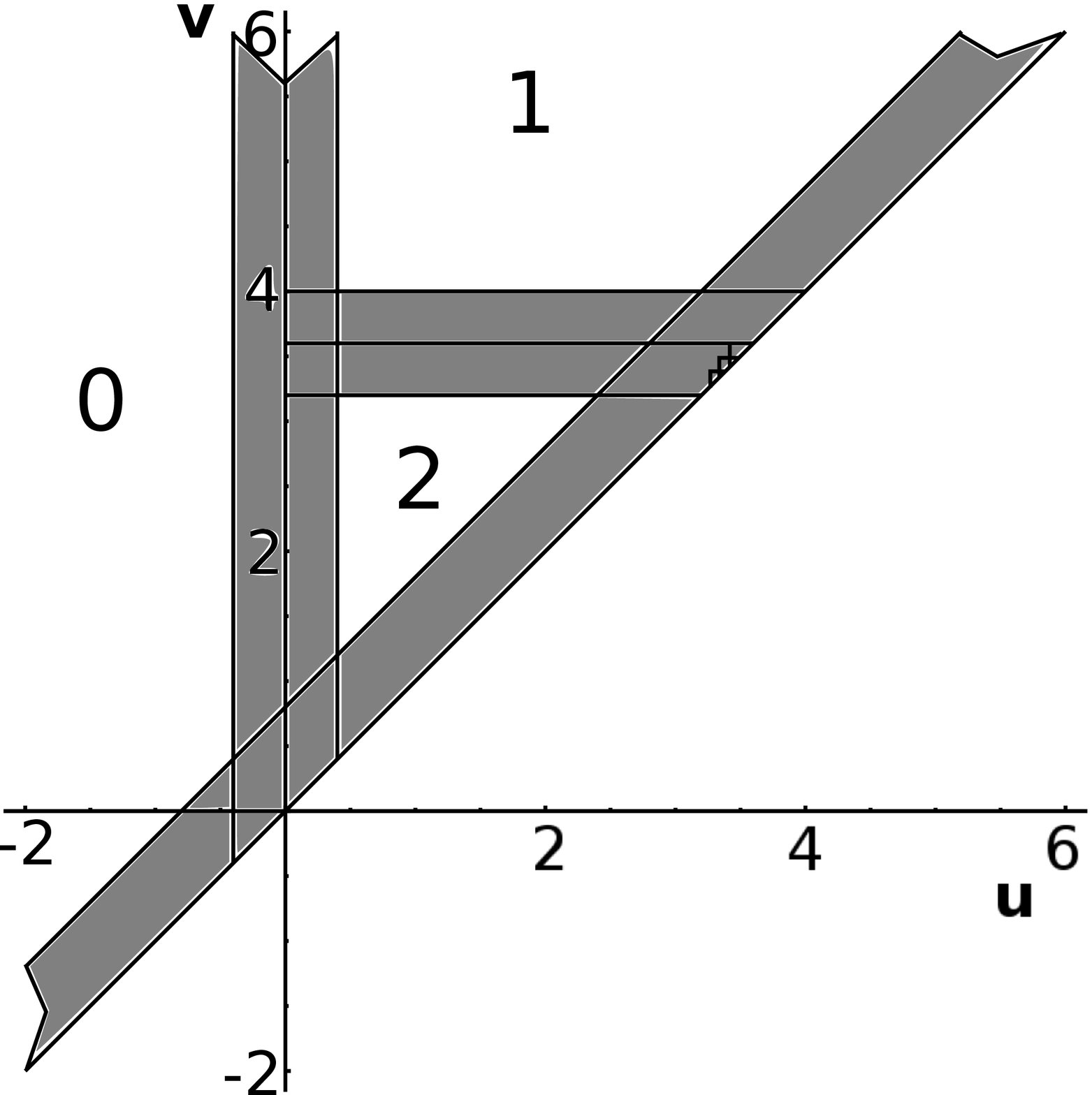}
     \caption{The 0-PBNs of the ball covering of $X$, with blind strips.} \label{circleh0}
  \end{minipage}%
  \hfill%
  \begin{minipage}[ht]{0.5\linewidth}
     \centering
     \includegraphics[scale=0.2]{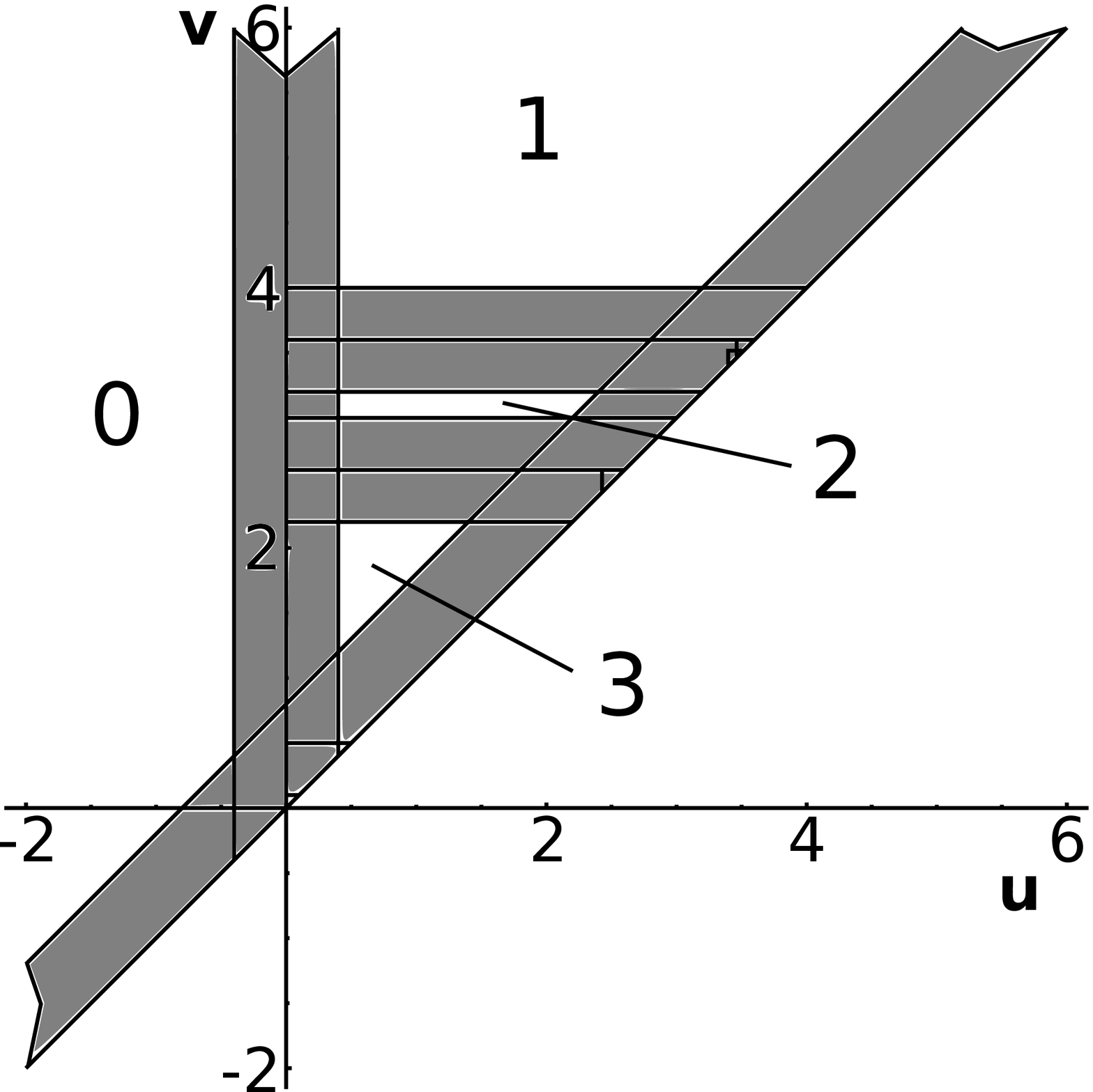}
     \caption{The 0-PBNs of the ball covering of $Y$, with blind strips.} \label{beanh0}
  \end{minipage}
\end{figure}

There is a triangle, not covered by blind strips in both diagrams, where the PBN for $X$ is 2 and
for $Y$ is 3 by Theorem \ref{wrap}. So we have
\[\beta_{(Y,g,0)}(0.4, 2.2) = 3 > 2 = \beta_{(X,f,0)}(1.1, 1.5)\]
and, by Theorem \ref{bound}, $\delta\left((X, f), (Y, g)\right) \ge 0.7$.

If we made use of smaller - but denser - balls of radius 0.2, then we would have
\[\beta_{(Y,g,0)}(0.2, 2.6) = 3 > 2 = \beta_{(X,f,0)}(1.3, 1.5)\]
and $\delta\left((X, f), (Y, g)\right) \ge 1.1$. (The true natural pseudodistance of the two pairs can be computed directly from the definition, and equals 1.5.)

\section{A combinatorial representation}

The ball unions of Section \ref{main}, although generated by finite sets,
are still continuous objects. It is desirable that the topological information
on $X$, up to a certain approximation, be condensed in a combinatorial object.
For size functions (i.e. for PBNs of degree 0) it was a graph; here, it has to be a simplicial complex.
We shall build such a complex, by following~\cite{Ed95}, to which we refer for
all definitions not reported here. Please note that~\cite{Ed95} uses {\it weighted}
Voronoi cells and diagrams, while we do not need to worry about that, since all
of our balls have the same radius; so the customary Euclidean distance can be used
instead of the power distance employed in that paper.

Let $X$, $L=\{l_1,\ldots,l_k\}$ and $\delta$ be as in Section \ref{app} (the case of Section \ref{point}
is an immediate extension). Moreover, let the points of $L$ be in general position.
For each $l_j\in L$, let $B_j=B(l_j, \delta)$
be the ball of radius $\delta$, centered at $l_j$. The set $B=\{B_1,\ldots,B_k\}$ is
a ball covering of $X$; denote by $U$ the corresponding ball union.
Let now $V_j$ be the {\it Voronoi cell} of $B_j$, i.e. the set of points
of $\mathbb{R}^m$ whose distance from $l_j$ is not greater than the distance from any other $l_{j'}$.

The set ${\cal V} = \{V_1,\ldots,V_k\}$ is the {\it Voronoi diagram} of $B$.
From $\cal V$ we get the collection of cells ${\cal Q}=\{V'_j = V_j \cap B_j \,\vert\, j=1,\ldots,k\}$,
a decomposition of $U$.

The {\it nerve} $N({\cal Q})$ of $\cal Q$ is the abstract simplicial complex where vertices
are the elements of $\cal Q$ and, for a subset $T$ of $\{1,\ldots,k\}$, the set of vertices
$\{V'_j \,\vert\, j\in T\}$ is a simplex if and only if $\bigcap_{j\in T} V'_j \neq \emptyset$.

For any $T \subseteq \{1,\ldots,k\}, T\neq \emptyset$ we denote by $\sigma_T$ the convex hull of $\{l_j \,\vert\, j \in T\}$

The {\it dual complex} of $\cal Q$ is ${\cal K} = \{\sigma_T \,\vert\, \{V'_j \,\vert\, j\in T\} \in N({\cal Q})\}$
and ${\cal S} = \vert {\cal K} \vert$, union of the simplices of $\cal K$, is the {\it dual shape} of $U$.

For a better understanding of the previous part we produce a toy example. Let
$X$ be a quarter of circle of radius 4  and $U$ be the union of
nine balls of radius 1, with centers near $X$ (Figure \ref{figc1}).
The Voronoi Diagram $\cal V$ associated to this ball covering $B$ is depicted in
Figure \ref{figc2}.

\begin{figure}[ht]
  \centering
  \begin{minipage}[ht]{0.49\linewidth}
     \centering
     \includegraphics[scale=0.2]{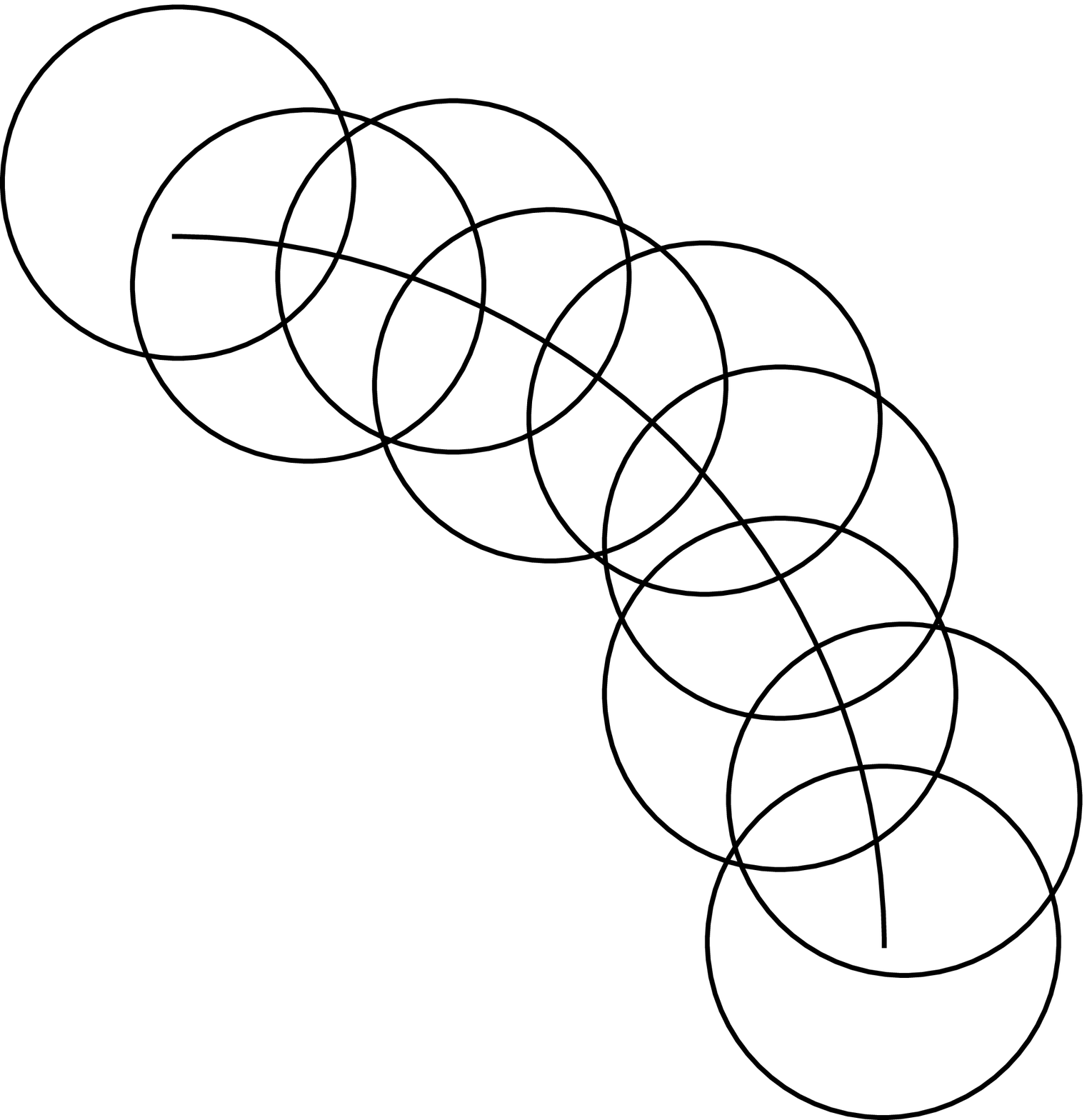}
     \caption{A quarter of circle of radius 4 covered by nine balls of radius 1.} \label{figc1}
  \end{minipage}%
  \hfill%
  \begin{minipage}[ht]{0.51\linewidth}
     \centering
     \includegraphics[scale=0.2]{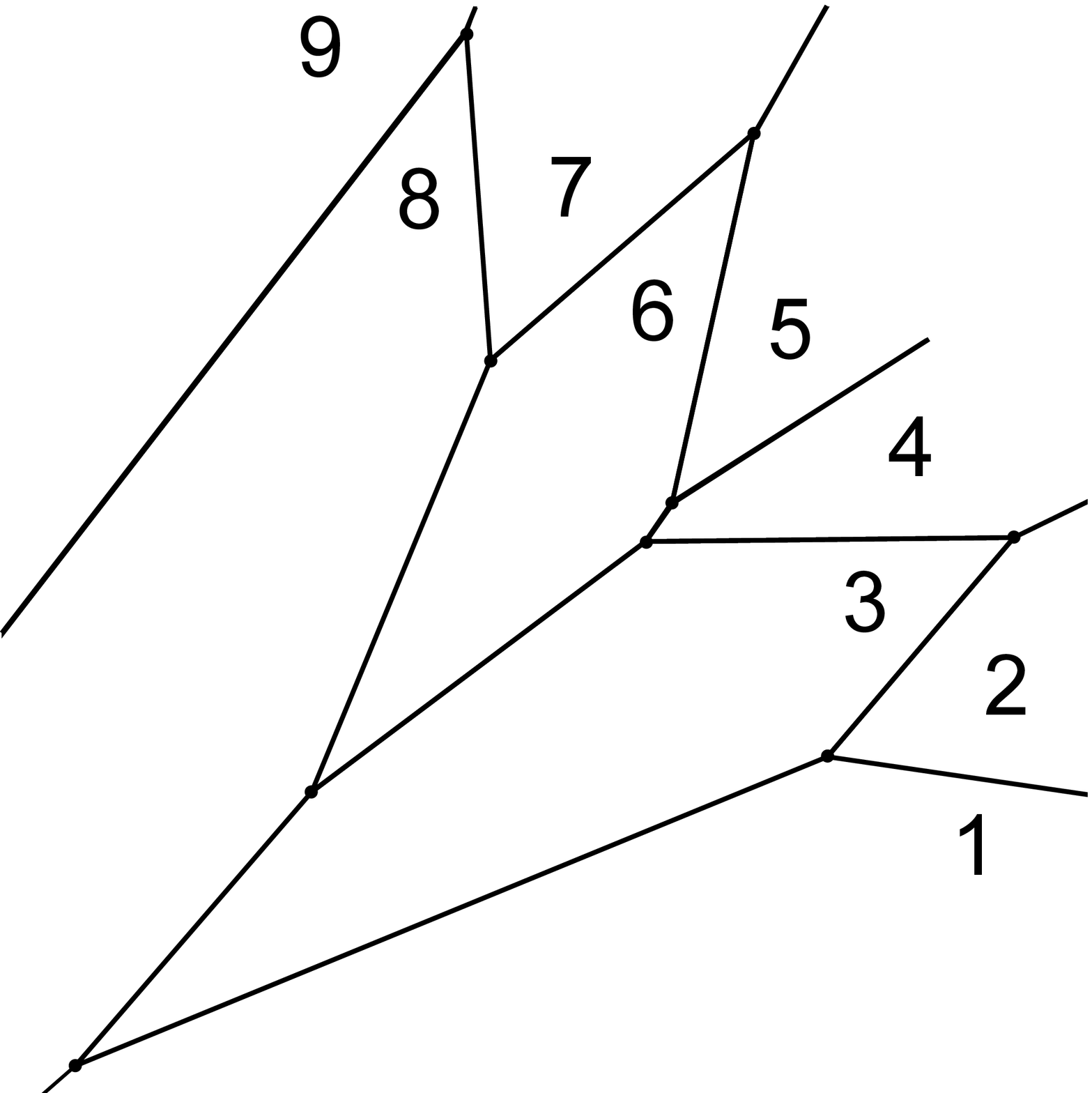}
     \caption{The Voronoi Diagram $\cal V$ of $B$.} \label{figc2}
  \end{minipage}
\end{figure}

Now the main idea is that we can associate the dual complex
$\cal K$ with the submanifold $X$. In fact, by Theorem 3.2 of~\cite{Ed95},
its space $\cal S$ is homotopically equivalent to $U$ and, by transitivity, to $X$.
Moreover, Section 3 of~\cite{Ed95} explicitly builds a
retraction $r$ from $U$ to $\cal S$ and a homotopy $H$ from the identity of $U$, to $p$,
such that $\forall y \in {\cal S}$, $\forall v \in
p^{-1}(y)$, $\forall t\in I$ we have $(p\circ H)(v,t)=y$.
For a complete description of the homotopy $H$ and the retraction
$p$ we refer to the original article.

$\cal K$ and $\cal S$ are shown in Figures
\ref{figc4},
\ref{figc5} respectively.

\begin{figure}[ht]
  \centering
\begin{minipage}[ht]{0.5\linewidth}
     \centering
     \includegraphics[scale=0.17]{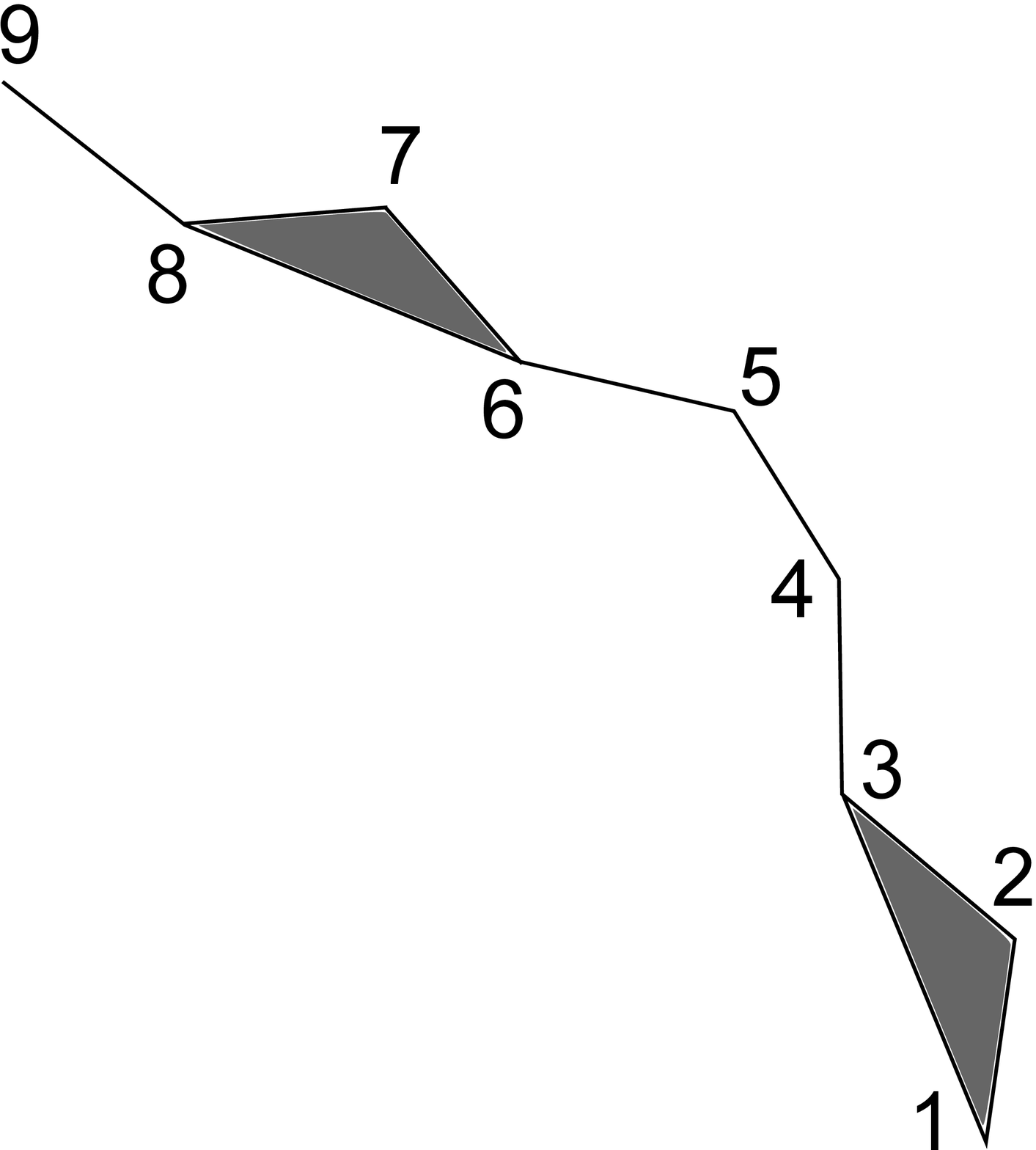}
     \caption{The dual complex $\cal K$.} \label{figc4}
  \end{minipage}%
  \hfill%
  \begin{minipage}[ht]{0.5\linewidth}
     \centering
     \includegraphics[scale=0.2]{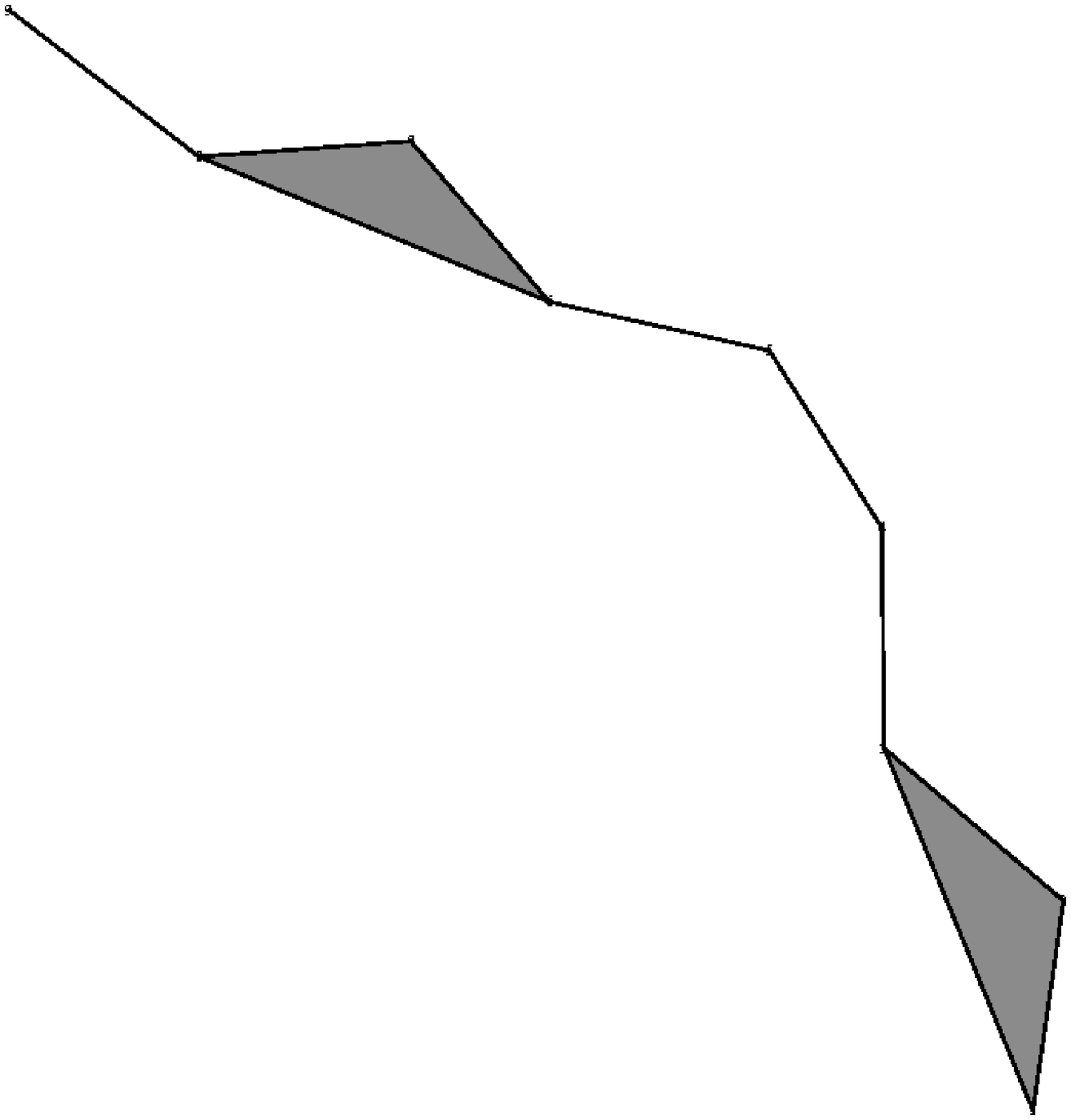}
     \caption{The dual shape $\cal S$.} \label{figc5}
  \end{minipage}
\end{figure}

\subsection{Ball union and dual shape}

Let $ f:\mathbb R^m \to \mathbb R^n$ be a
continuous function and let $ f_{\cal S}$ and $ f_U$ be the
restrictions of $ f$ to $\cal S$ and $U$ respectively. The notation for the PBN's $\beta$ is simplified as before.

\begin{lemma}\label{t4}
If $( u,  v)$ is
a point of $\Delta^+$ and if $ u + \omega(\delta) \prec  v
- \omega(\delta)$, where $
\omega(\delta)=(\Omega(\delta),\ldots,\Omega(\delta)) \in
\mathbb{R}^n$ , then
\[\beta_U( u- \omega(\delta),  v+ \omega(\delta))
\leq \beta_{\cal S}( u, v)\leq \beta_U( u+ \omega(\delta), v- \omega(\delta)).\]

If $\overline{u}, \overline{v} \in \mathbb{R}^n$ are such that
\[\beta_U( \overline{u}, \overline{v}) \neq \beta_{\cal S}( \overline{u}, \overline{v})\]
then there is at least a point $(\tilde{u}, \tilde{v})$ with max norm
$\| (\overline{u}, \overline{v}) - (\tilde{u}, \tilde{v}) \| \le \Omega(\delta)$, which is either an element of the boundary of $\Delta^+$ or a discontinuity
point of $\beta_U$.
\end{lemma}

\begin{proof}
By Lemma \ref{lemmas3}, with $Y={\cal S}$, $V=U$.
\end{proof}

Now we can get an estimate of the PBNs of $X$ from the ones of $\mathcal{S}$.
The blind strips of the 1D reduction will be doubly wide, with respect to the ones previously considered.
Still, this can leave some regions of $\Delta^+$ where the computation is exact. Also here the position of the blind strips
is well determined by the position of the discontinuity lines of the PBNs of $\mathcal{S}$.

\begin{theorem}\label{t5}
If $( u,  v)$ is
a point of $\Delta^+$ and if $ u + 2 \omega(\delta) \prec  v
-2 \omega(\delta)$, where $
\omega(\delta)=(\Omega(\delta),\ldots,\Omega(\delta)) \in
\mathbb{R}^n$ , then
\[\beta_{\cal S}( u-2 \omega(\delta),  v+2 \omega(\delta))
\leq \beta_X( u, v)\leq \beta_{\cal S}( u+2 \omega(\delta), v-2 \omega(\delta)).\]

If $\overline{u}, \overline{v} \in \mathbb{R}^n$ are such that
\[\beta_X( \overline{u}, \overline{v}) \neq \beta_{\cal S}( \overline{u}, \overline{v})\]
then there is at least a point $(\tilde{u}, \tilde{v})$ with max norm
$\| (\overline{u}, \overline{v}) - (\tilde{u}, \tilde{v}) \| \le 2\Omega(\delta)$, which is either an element of the boundary of $\Delta^+$ or a discontinuity
point of $\beta_{\cal S}$.
\end{theorem}

\begin{proof} By Theorem \ref{wrap}, and with $\omega=\omega(\delta)$,
\[\beta_U( u- \omega,  v+ \omega)\leq \beta_X( u, v)\leq \beta_U( u+ \omega, v- \omega)\]
Then we have
\[\beta_U( u+ \omega, v- \omega)
\leq \beta_{\cal S}( u+2 \omega, v-2 \omega)\]
\noindent by Lemma \ref{t4} by substituting $( u, v)$ with
$( u +2\omega, v-2\omega)$, and
\[\beta_{\cal S}( u-2 \omega,  v+2 \omega)
\leq \beta_U( u- \omega, v+ \omega)\]
\noindent by Lemma \ref{t4} by substituting $( u, v)$ with
$( u -2\omega, v+2\omega).$

The second part can be proved in an analogous way from the second part of Lemma \ref{t4}.
\end{proof}

\subsection{An example in 2D persistence}

As a simple example, we now apply Theorems \ref{bound} and \ref{t5} to two pairs $(X, f)$ and $(Y, g)$ with
2-dimensional filtering functions. Let $X$ and $Y$ be two circles of radius $r$ embedded in $\R^2$ and $f, g$
be (unknown) continuous functions from $\R^2$ to $\R^2$, whose restrictions to suitable neighbourhoods of $X, Y$
have modulus of continuity $\Omega(\delta) = \frac{\delta}{2 r}$. Here the range $\R^2$ parametrizes the plane
$\Pi$ of equation $R+G+B = \frac{4-\sqrt{2}}{2}$ (chosen as one which contains a fairly large square of points with nonnegative
coordinates) of the color space $RGB$. A Cartesian reference frame $x, y, z$ has been fixed in the $RGB$ space, so
that $\Pi$ is the $z=0$ plane. The change of reference is:

$$ \left(
\begin{matrix} R\cr G\cr B\end{matrix}\right) =
\left(\begin{matrix} -{1}\over{\sqrt{2}}&-{1}\over{\sqrt{6}}&{1}\over{\sqrt{3}}\cr
{1}\over{\sqrt{2}}&-{1}\over{\sqrt{6}}&{1}\over{\sqrt{3}}\cr
0&\sqrt{{2}\over{3}}&{1}\over{\sqrt{3}}\end{matrix}\right)\cdot
\left(\begin{matrix} x\cr y\cr z\end{matrix}\right) +
\left(\begin{matrix} {4-\sqrt{2}}\over{4}\cr {4-\sqrt{2}}\over{4}\cr 0\end{matrix}\right)$$

Of the functions on two circles we know a sampling given by 80 regularly spaced points, represented in
Figure \ref{color}.

\begin{figure}[htbp]
  \centering
  \begin{minipage}[ht]{0.9 \linewidth}
     \centering
     \includegraphics[scale=0.3]{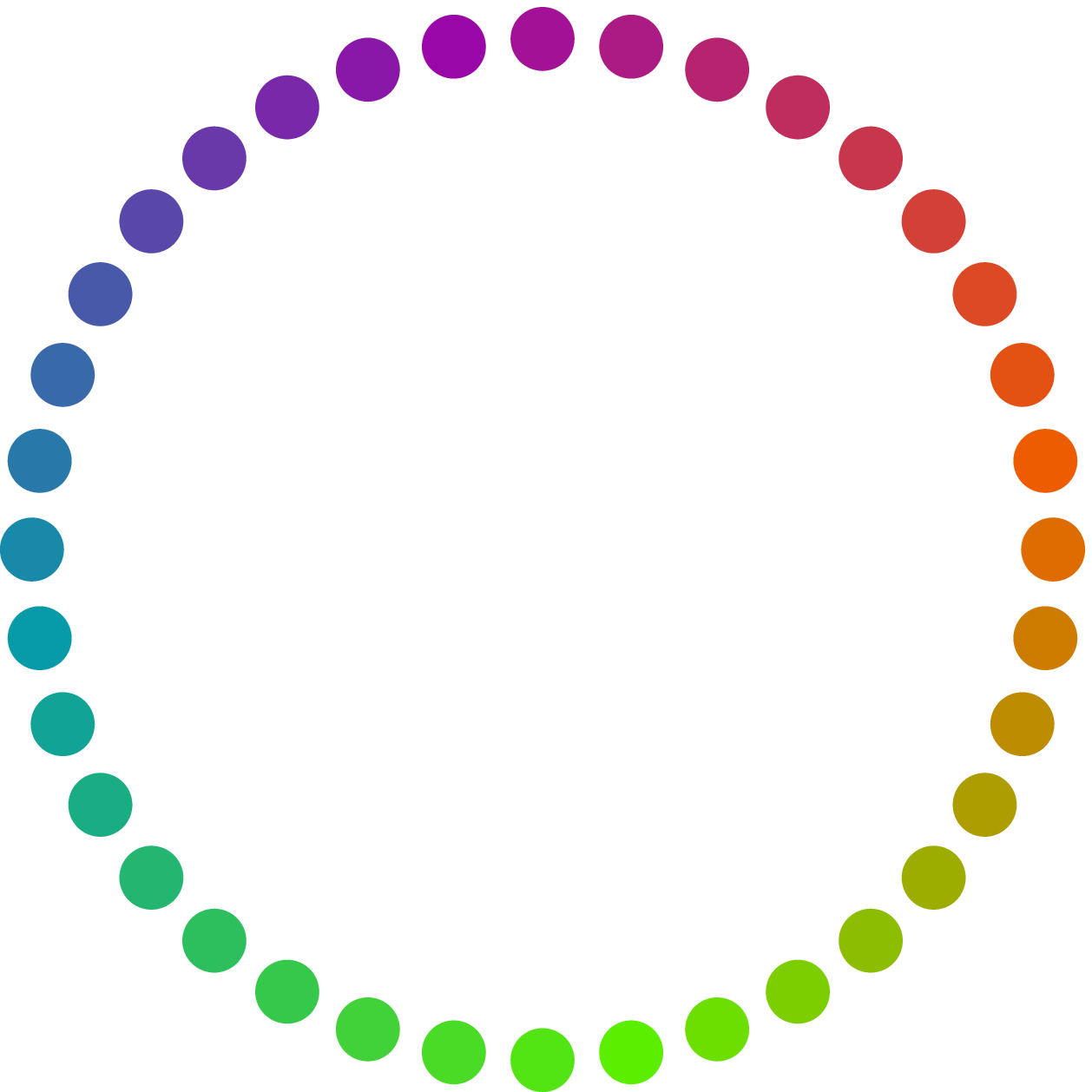}
     \includegraphics[scale=0.3]{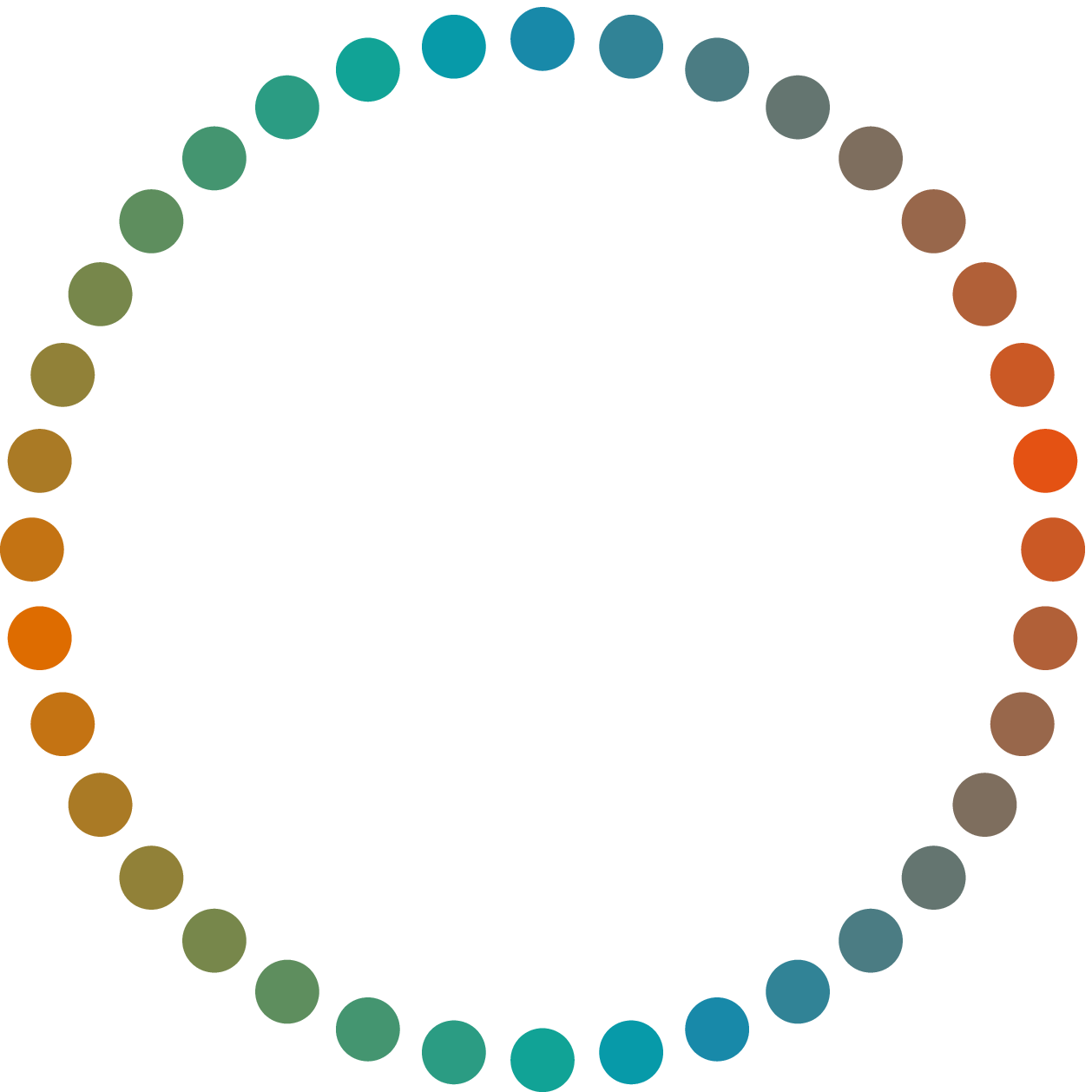}
     \caption{The color samplings of $(X, f)$ (left) and $(Y, g)$ (right).} \label{color}
  \end{minipage}%
 \end{figure}

We can think of the 80 points on each circle as centers of disks of radius $\delta = 0.08 r$. The distance between
two consecutive points is $2 r \pi \sin(\frac{\pi}{80}) \in ]0,078 r, 0.079 r[$, so the balls form a covering
which respects the hypotheses of Proposition \ref{smale} since $\tau = r$ and $0.08 < \sqrt{\frac{3}{5}}$. The corresponding
dual complex is a 1D cycle and the dual shape is a closed polygonal.

The point images in the color plane $\Pi$ are along the edges of the square of vertices $(-0.4, 0),
(-0.4, 0.8), (0.4, 0.8), (0.4, 0)$ for $(X, f)$ and along the polygonal of vertices $(-0.4, 0.04), (0.36, 0.8),
(0.4, 0.8), (0.4, 0.76), (-0.36, 0)$ (twice) for $(Y, g)$, two consecutive ones at distance $0.04 = \Omega(\delta)$
in the max norm.

A 1D reduction of the PBNs is possible through a foliation of the domain $\R^2 \times \R^2$.
In the terminology of~\cite{BiCeFrGiLa,CaDiFe}, a suitable admissible pair $\vec{l}, \vec{b}$ is given by
$\vec{l} = (\frac{1}{\sqrt{2}}, \frac{1}{\sqrt{2}})$ and $\vec{b} = (-0.2, 0.2)$; in the corresponding leaf
of the foliation, the PBN's of the dual shapes and the blind strips according to Theorem \ref{t5}
are as shown in Figure \ref{colPBN}.

\begin{figure}[htbp]
  \centering
  \begin{minipage}[ht]{1.0 \linewidth}
     \centering
     \includegraphics[scale=0.3]{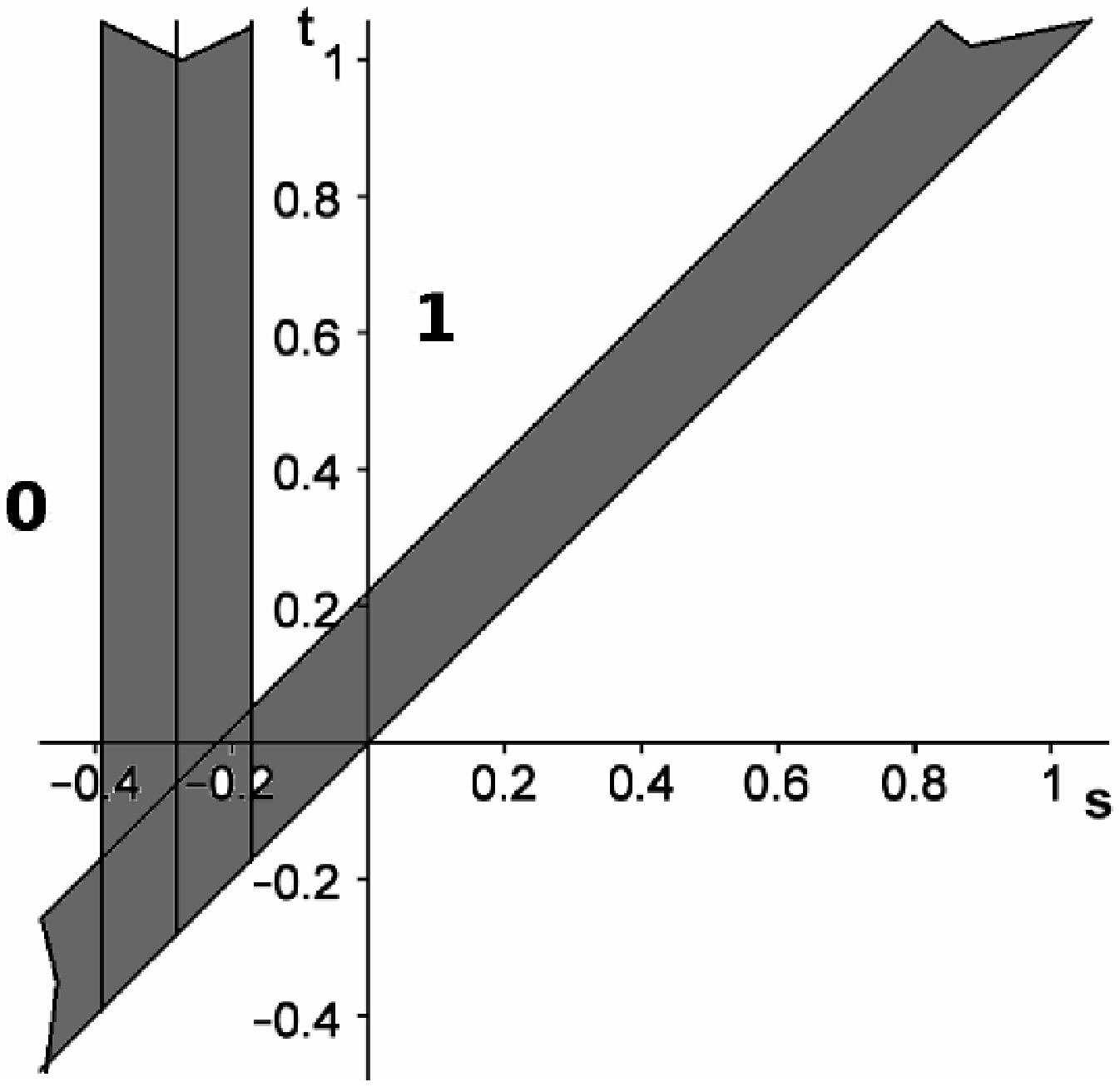}
     \includegraphics[scale=0.3]{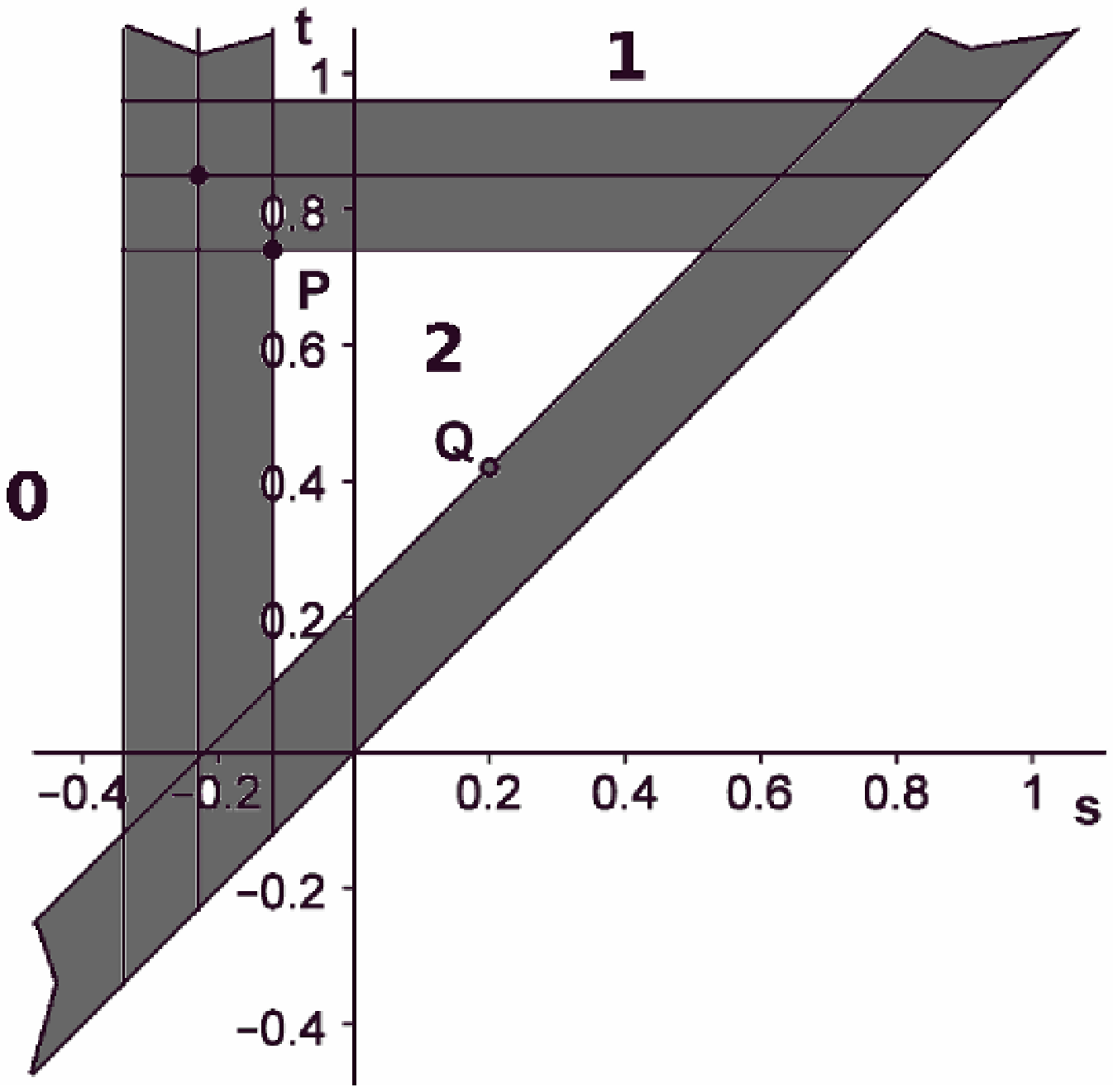}
     \caption{PBNs of the two colored circles} \label{colPBN}
  \end{minipage}%
 \end{figure}

In correspondence of points $P, Q$ in the leaf, we get $(u, v), (u', v') \in \R^2\times\R^2$ respectively,
with $u=(-0.28, 0.12)$, $v=(0.32, 0.72)$, $u'=(-0.06, 0.34)$, $v'=(0.1, 0.5)$. Then we have
\[\beta_{(Y,g,0)}(u, v) = 2 > 1 = \beta_{(X,f,0)}(u', v')\]
and, by Theorem \ref{bound}, $\delta\left((X, f), (Y, g)\right) \ge 0.22$.

\subsection*{Acknowledgements}
The authors wish to thank P. Frosini for the many
helpful suggestions. This work was performed under the auspices of INdAM-GNSAGA and
ARCES.

\bibliographystyle{abbrv}
\bibliography{bibliostrips}

\end{document}